\documentclass[11pt]{amsart}
\usepackage{amsthm,amssymb, amsmath}

\usepackage{graphicx, float, epsfig}
\usepackage[nobysame,alphabetic]{amsrefs}
\usepackage{tikz}
\usepackage{bbm}

\usetikzlibrary{arrows}
\usetikzlibrary{snakes}

\renewcommand{\a}{\alpha}
\renewcommand{\b}{\beta}
\newcommand{\e}{\epsilon}

\newcommand\w{\omega}
\newcommand\N{\mathbb N}
\newcommand\Z{\mathbb Z}
\newcommand\R{\mathbb R}

\renewcommand\P{\mathbb{P}}
\newcommand\E{\mathbb{E}}

\renewcommand\le{\leqslant}
\renewcommand\ge{\geqslant}

\newcommand{\cross}{\times}
\newcommand{\norm}[1]{\left|#1\right|}
\newcommand{\Xh}{\overline{X}^h}
\newcommand{\Xb}{\partial X}
\newcommand{\Xg}{X \cup \Xb}

\newcommand{\xq}{\Xh_F}
\newcommand{\xr}{\Xh_\infty}

\newcommand{\seq}[1]{(#1)_{n \in \N}}
\newcommand{\seqz}[1]{(#1)_{n \in \Z}}

\newcommand{\gp}[3]{{(#2\cdot#3)_{#1}}}

\newcommand\bg{\mathcal{O}}
\newcommand\rmu{\check \mu}
\newcommand\rnu{\check \nu}
\newcommand\tmu{\widetilde \mu}
\newcommand\tnu{\widetilde \nu}

\newcommand\mun[1]{ { \mu_{#1} } }
\newcommand\tmun[1]{\tmu_{#1}}

\newcommand\supp{\text{supp}}

\newtheorem{theorem}{Theorem}[section]
\newtheorem{corollary}[theorem]{Corollary}

\newtheorem{lemma}[theorem]{Lemma}
\newtheorem{proposition}[theorem]{Proposition}
\newtheorem*{prop}{Proposition \ref{prop:hyperbolic}}

\newtheorem{definition}[theorem]{Definition}

\newtheorem{claim}[theorem]{Claim}

\newtheorem{example}[theorem]{Example}

\title{Random walks on weakly hyperbolic groups}
\author{Joseph Maher, Giulio Tiozzo}

\begin{document}

\begin{abstract}
Let $G$ be a countable group which acts by isometries on a separable,
but not necessarily proper, Gromov hyperbolic space $X$.  We say the
action of $G$ is weakly hyperbolic if $G$ contains two independent
hyperbolic isometries.  We show that a random walk on such $G$ converges to
the Gromov boundary almost surely. We apply the convergence result to
show linear progress and linear growth of translation length, without
any assumptions on the moments of the random walk. 

If the action is acylindrical, and the random walk has finite entropy
and finite logarithmic moment, we show that the Gromov boundary with
the hitting measure is the Poisson boundary.
\end{abstract}

\maketitle


\section{Introduction}

We say a geodesic metric space $X$ is \emph{Gromov hyperbolic}, or
\emph{$\delta$-hyperbolic}, if there is a number $\delta \ge 0$ for
which every geodesic triangle in $X$ satisfies the \emph{$\delta$-slim
  triangle} condition, i.e. any side is contained in a
$\delta$-neighbourhood of the other two sides. Throughout this paper
we will assume that the space $X$ is separable, i.e. it contains a
countable dense set, but we will not assume that $X$ is proper or
locally compact.  For example, any countable simplicial complex
satisfies these conditions.

Let $G$ be a countable group which acts by isometries on $X$.  We say
the action of $G$ on $X$ is \emph{non-elementary} if $G$ contains a
pair of hyperbolic isometries with disjoint fixed points in the Gromov
boundary. We say $G$ is \emph{weakly hyperbolic} if it admits a
non-elementary action by isometries on some Gromov hyperbolic space
$X$.  In this case, a natural boundary for the group is given by the
Gromov boundary $\partial X$ of $X$, which however need not be
compact.

Several widely studied group actions are weakly hyperbolic in this
sense, in particular:

\begin{itemize}
\item Hyperbolic and relatively hyperbolic groups; 
\item Mapping class groups, acting on the curve complex; 
\item Out($F_n$) acts on various Gromov hyperbolic simplicial
complexes, for example the complex of free factors or the complex of
free splittings;
\item Right-angled Artin groups acting on their extension graphs; 
\item Finitely generated subgroups of the Cremona group.
\end{itemize}
In particular, all acylindrically hyperbolic groups are weakly
hyperbolic, see Section \ref{section:discussion} for further
discussion and more examples.

In this paper, we shall consider random walks on weakly hyperbolic groups, 
constructed by choosing products of random group elements.
A probability distribution $\mu$ on $G$ determines a
random walk on $G$, by taking the product
$$w_n := g_1 g_2 \dots g_n$$
where the $g_i$ are independent identically distributed elements of
$G$, with distribution $\mu$.  A choice of basepoint $x_0 \in X$
determines an orbit map sending $g \mapsto g x_0$, and we can project
the random walk on $G$ to $X$ by considering the sequence $(w_n
x_0)_{n \in \mathbb{N}}$, which we call a \emph{sample path}.
We say a measure $\mu$ on $G$ is \emph{non-elementary} if 
the semi-group generated by its support is a non-elementary subgroup of $G$. 

\subsection{Results}

The first result we establish is that sample paths converge almost surely in the Gromov boundary:

\begin{theorem} \label{theorem:convergence} %
Let $G$ be a countable group of isometries of a
separable Gromov hyperbolic space $X$, and let $\mu$ be a 
non-elementary probability
distribution on $G$. 
Then, for any basepoint $x_0 \in X$, almost every sample path
$(w_n x_0)_{n \in \N}$ converges to a point
$\omega_+ \in \partial X$.
The resulting hitting measure $\nu$ is non-atomic, and is the
unique $\mu$-stationary measure on $\partial X$.
\end{theorem}

We use the convergence to the boundary result to show the following
linear progress, or positive drift, result.
Recall that a measure $\mu$ has finite first moment if $\int_G d_X(x_0, g x_0) \ d\mu(g) < \infty$.

\begin{theorem} \label{theorem:linear progress} %
Let $G$ be a countable group of isometries of a separable
Gromov hyperbolic space $X$, and let $\mu$ be a non-elementary
probability distribution on $G$, and $x_0 \in X$ a basepoint.
Then there
is a constant $L > 0$ such that for almost every
sample path we have 
\[ \liminf_{n \to \infty} \frac{d_X(x_0, w_n x_0)}{n} = L > 0. \]
Furthermore, if $\mu$ has finite first moment, then the limit
%
\[ \lim_{n \to \infty} \frac{d_X(x_0, w_n x_0)}{n}  = L > 0   \]
exists almost surely.
Finally, if the support of $\mu$ is
bounded in $X$, then there are constants $c < 1, K$ and $L > 0$ such that
\[ \P( d_X(x_0, w_n x_0) \le Ln ) \le K c^n  \]
for all $n$.
\end{theorem}

Note that the constants $L$, $K$, and $c$ depend on the choice of the measure $\mu$. Moreover,  the first statement also implies that 
\[ \lim_{n \to \infty} \P ( d_X(x_0, w_n x_0) \le Ln ) = 0 \]
(for a possibly different constant $L$). If we assume that $\mu$ has finite first moment with respect to the
distance function $d_X$, then we obtain the following geodesic
tracking result.

\begin{theorem} \label{theorem:sublinear} %
Let $G$ be a countable group which acts by isometries on a separable
Gromov hyperbolic space $X$ with basepoint $x_0$, and let $\mu$ be 
non-elementary 
probability distribution on $G$, with finite first moment.
Then for almost every sample path $\seq{w_n x_0}$ there is a quasigeodesic
ray $\gamma$ 
which tracks the sample path sublinearly, i.e.
\[  \lim_{n \to \infty} \frac{ d_X( w_n x_0, \gamma) }{n} = 0, \text{
  almost surely.}   \]
If the support of $\mu$ is bounded in $X$, then in fact the tracking
is logarithmic,  i.e.
\[  \limsup_{n \to \infty} \frac{ d_X( w_n x_0, \gamma) }{ \log n} <
\infty, \text{ almost surely.}  \]
\end{theorem}

Finally, we investigate the growth rate of translation length of group
elements arising from the sample paths.  

\begin{theorem} \label{theorem:translation} %
Let $G$ be a countable group which acts by isometries on a separable
Gromov hyperbolic space $X$, and let $\mu$ be a
non-elementary probability distribution on $G$.
Then the translation length $\tau(w_n)$ of the group element $w_n$ grows at least
linearly in $n$, i.e.
\[ \P( \tau(w_n) \le L n ) \to 0 \text{ as } n \to \infty,  \]
for some constant $L$ strictly greater than zero.

If the support of $\mu$ is bounded in $X$, then there are constants
$c < 1, K$ and $L > 0$, such that
\[ \P( \tau(w_n) \le Ln ) \le K c^n \]
for all $n$.
\end{theorem}

Recall that the translation length $\tau(g)$ of an isometry $g$ of $X$ is defined as
\[ \tau(g) := \lim_{n \to \infty} \tfrac{1}{n} d_X(x_0, g^n x_0). \]
As an element with non-zero translation length is a hyperbolic (= loxodromic)
isometry, this shows that the probability that a random walk of length
$n$ gives rise to a hyperbolic isometry tends to one as $n$ tends to
infinity.

\medskip
\emph{The Poisson boundary of acylindrically hyperbolic groups.}
A special class of weakly hyperbolic groups are the \emph{acylindrically
hyperbolic} groups. In this case, we show that we may identify the Gromov boundary 
$(\Xb, \nu)$ with the Poisson boundary.

Recall that a group $G$ acts acylindrically on a Gromov hyperbolic
space $X$, if for every $K \ge 0$ there are numbers $R$ and $N$, which
both depend on $K$, such that for any pair of points $x$ and $y$ in
$X$, with $d_X(x, y) \ge R$, there are at most $N$ group elements $g$
in $G$ such that $d_X(x, gx) \le K$ and $d_X(y, gy) \le K$.

\begin{theorem} \label{theorem:poisson} %
Let $G$ be a countable group of isometries which acts acylindrically
on a separable Gromov hyperbolic space $X$, let $\mu$ be a
non-elementary probability distribution on $G$ with finite entropy and
finite logarithmic moment, and let $\nu$ be the hitting measure on
$\Xb$.  Then $(\partial X, \nu)$ is the Poisson boundary of $(G,
\mu)$.
\end{theorem}

\subsection{Examples and discussion} \label{section:discussion}

\qquad \newline

\emph{Word hyperbolic groups}. The simplest example of a weakly hyperbolic group is a (finitely generated) 
Gromov hyperbolic group acting on its Cayley graph, which by definition is a
$\delta$-hyperbolic space, and so any non-elementary Gromov hyperbolic
group is weakly hyperbolic.  
Convergence of sample paths to the Gromov boundary in this case is due to Kaimanovich 
\cite{kai94}, who also shows that the Gromov boundary may be identified with the Poisson boundary
and the hitting measure is the unique $\mu$-stationary measure.
Note that in this case the space is locally compact, and the boundary is compact.


\emph{Relatively hyperbolic groups.} The Cayley graph of a relatively hyperbolic group is
$\delta$-hyperbolic with respect to an infinite generating set, and so
these groups are also weakly hyperbolic, but in this case the space on
which the group acts need not be proper. In this case, convergence to
the boundary was shown by Gautero and Math{\'e}us \cite{gm}, who also
covered the case of groups acting on $\R$-trees. More recently, group actions on 
(locally infinite) trees have been considered by Malyutin and Svetlov \cite{ms}.

There are then groups which are weakly hyperbolic, but not relatively
hyperbolic, the two most important examples being the mapping class
groups of surfaces, and Out($F_n$).

\emph{Mapping class groups.}
The mapping class group Mod($S$) of a surface $S$ of genus $g$ with
$p$ punctures acts on the curve complex $C(S)$, which is a locally
infinite simplicial complex.  As shown by Masur and Minsky \cite{mm1},
the curve complex is $\delta$-hyperbolic, and moreover, the action is
acylindrical, by work of Bowditch \cite{bowditch}, so we can apply our
techniques to get convergence and the Poisson boundary.

Convergence to the boundary of the curve complex also follows from
work of Kaimanovich and Masur \cite{km} and Klarreich
\cite{klarreich}, using the action of Mod($S$) on Teichm\"uller space
(which is locally compact, but not hyperbolic).  Indeed, Kaimanovich
and Masur show that random walks on the mapping class group converge
to points in Thurston's compactification of Teichm\"uller space
$\mathcal{PMF}$, and then Klarreich (see also Hamenst\"adt
\cite{hamenstadt}) shows the relation between $\mathcal{PMF}$ and the
boundary of the curve complex.  Our approach does not use fine
properties of Teichm\"uller geometry.  A third approach is to consider
the action of Mod($S)$ on Teichm\"uller space with the Weil-Petersson
metric, which is a non-proper CAT(0) space.  By work of Bestvina,
Bromberg and Fujiwara \cite{bbf}, this space has finite telescopic
dimension, and one may then apply the results of Bader, Duchesne and
L\'ecureux \cite{bdl}.

We remark that $C(S)$ does not possess a CAT(0) metric, since it is
homotopic to a wedge of spheres (Harer \cite{harer}); in fact,
Kapovich and Leeb \cite{kapovich-leeb} showed that the mapping class
group (of genus at least $3$) does not act freely cocompactly on a
CAT(0) space, though it is still open as to whether there is a proper
CAT(0) space on which the mapping class group acts by isometries;
Bridson \cite{bridson} showed that any such action must have elliptic
or parabolic Dehn twists.

\emph{Out($F_n$).}
The outer automorphism group of a non-abelian free group, Out($F_n$),
acts on a number of distinct Gromov hyperbolic spaces, as shown by
Bestvina and Feighn \cites{bestvina-feighn, bestvina-feighn2} and
Handel and Mosher \cite{hm}, and so is weakly hyperbolic. 
Similarly to the case of Mod($S$), convergence
to the boundary also follows by considering 
the action of Out($F_n$) on the (locally compact)
outer space, as shown by Horbez \cite{horbez}.

\emph{Right-angled Artin groups.}  
A right-angled Artin group acts by simplicial isometries on its
extension graph, which has infinite diameter as long as the group does
not split as a non-trivial direct product, and is not quasi-isometric
to $\Z$. Kim and Koberda showed that the extension graph is a
(non-locally compact) quasi-tree \cite{kim-koberda}, and in fact the
action is acylindrical \cite{kim-koberda2}.

\emph{Finitely generated subgroups of the Cremona group.} %
Manin \cite{manin} showed that the Cremona group acts faithfully by
isometries on an infinite-dimensional hyperbolic space, known as the
Picard-Manin space, which is not separable. However, any
finite-generated subgroup preserves a totally geodesic closed
subspace, which is separable, see for example Delzant and Py
\cite{delzant-py}.

\emph{Acylindrically hyperbolic groups.} %
The definition of an acylindrical group action is due to Sela
\cite{sela} for trees, and Bowditch \cite{bowditch} for general metric
spaces, see Osin \cite{osin} for a discussion and several examples of
acylindrical actions on hyperbolic spaces.  As every acylindrically
hyperbolic group is also weakly hyperbolic, this gives a number of
additional examples of weakly hyperbolic groups which are not
necessarily relatively hyperbolic; for example, all one relator groups
with at least three generators, see \cite{osin} for many other
examples.

\emph{Isometries of CAT(0) spaces.} Even though not all CAT(0) spaces are hyperbolic, the two theories overlap in many cases. 
For isometries of general CAT(0) spaces, Karlsson and Margulis \cite{km} 
proved boundary convergence and identified the Poisson boundary.
More recently, boundaries of
CAT(0) cube complexes have been studied by Nevo and Sageev \cite{nevo-sageev}, and
(not necessarily proper) CAT(0) spaces of finite telescopic dimension by Bader, Duchesne and L\'ecureux \cite{bdl}.

\medskip

Once we have proved convergence to the boundary, we apply
this to show positive drift.
In the locally compact case,
positive drift results go back to Guivarc'h \cite{guivarch}.
In particular, when the space on which $G$ acts is proper, 
positive drift follows from non-amenability of the group, 
but this need not be the case for non-proper spaces.
In the curve complex case, linear progress is due to Maher \cite{Maher_linear}.

We then show the sublinear tracking results, using work of Tiozzo
\cite{tiozzo}. Sublinear tracking can be thought of as a
generalization of Oseledec's multiplicative ergodic theorem
\cite{oseledec}. In our context, these results go back to Guivarc'h
\cite{guivarch}, and are known for groups of isometries of CAT(0) spaces by Karlsson and
Margulis \cite{km}, and for Teichm\"uller space by Duchin
\cite{duchin}. 
Sublinear tracking on
hyperbolic groups is due to Kaimanovich \cite{kai87, kai94}; 
moreover, 
Karlsson and Ledrappier
\cites{karlsson-ledrappier, karlsson-ledrappier2} proved a law 
of large numbers on general (proper) metric spaces using 
horofunctions.

Note that these results can be used to prove convergence to the boundary 
once one knows that the drift is positive. 
In the above-mentioned cases, the space is meant to be proper, so 
positive drift follows from non-amenability of the group, 
while a new argument is needed in general.

In this paper we give an argument for the non-proper weakly
hyperbolic case, where sublinear tracking and positive drift
\emph{follow} from convergence to the boundary.  Note that in our
approach we use horofunctions, and indeed
\cite{karlsson-ledrappier} can be used to simplify our proofs if
one assumes positive drift.
Recently (after the first version of this paper appeared), Mathieu and Sisto \cite{mathieu-sisto} 
 provided a different argument for positive drift in the acylindrical case.

Logarithmic tracking was previously
known for random walks on trees, due to Ledrappier \cite{ledrappier},
on hyperbolic groups, due to Blach\`ere, Ha\"issinsky and Mathieu
\cite{bhm}, and on relatively hyperbolic groups, due to Sisto
\cite{sisto2}.

Finally we show that the translation length grows linearly, which in
particular shows that the probability that a random walk gives rise to
a hyperbolic element tends to one. This generalizes earlier work of
Rivin \cite{rivin}, Kowalski \cite{kowalski}, Maher \cite{Mah} and
Sisto \cite{sisto}.

The methods in this paper build on previous work of Calegari and Maher
\cite{cm}, which showed convergence results with stronger conditions
on $X$ and $\mu$.

\subsection{Outline of the argument}

To explain the argument in the proof of Theorem \ref{theorem:convergence}, we briefly remind the reader of the standard argument for convergence
to the boundary for a random walk on a group $G$ acting on a locally
compact $\delta$-hyperbolic space $X$. 
The argument
ultimately goes back to Furstenberg \cite{Furstenberg}, who developed it for Lie groups.

\emph{Measures on the Gromov boundary.}
Let $\mu$ be the probability distribution on $G$ generating the random walk.
The first step is to find a $\mu$-stationary measure $\nu$ on
the Gromov boundary $\Xb$, and then apply the martingale convergence
theorem to show that for almost every sample path $\omega =
\seq{w_n}$, the sequence of measures $\seq{ w_n \nu }$ converges to
some measure $\nu_\omega$ in $\mathcal{P}(\Xb)$, the space of
probability measures on $\Xb$. One then uses geometric properties of
the action of $G$ on $X$ to argue that $\nu_\omega$ is a
$\delta$-measure $\delta_\lambda$ for some point $\lambda \in \Xb$,
almost surely, and that the image of the sample path under the orbit
map $\seq{w_n x_0}$ converges to $\lambda$.

This argument uses local compactness in an essential way in the first
step. For a locally compact hyperbolic space $X$, the Gromov
boundary $\Xb$ is compact, as is $X \cup \Xb$. The space of
probability measures on $\Xb$ is also compact, and so 
the existence of a $\mu$-invariant measure on $X$
just follows from taking weak limits.
In the non-locally compact case, the Gromov boundary $\Xb$, and $X
\cup \Xb$, need not be compact, as seen in the following example.

\begin{example}[Countable wedge of rays] \label{ex:wedge} %
A ray is a half line $\R_+ = \{ x \in \R : x \ge 0 \}$, with basepoint
$0$. Let $X$ be the wedge product of countably many rays.  This space
is a tree, and so is $\delta$-hyperbolic, and is not locally compact
at the basepoint. The Gromov boundary is homeomorphic to $\N$ with the
discrete metric, and is not compact.
\end{example}

\emph{The horofunction boundary.} %
In order to address this issue, we shall consider the horofunction
boundary of $X$, which was also initially developed by Gromov
\cite{bgs}, and has proved a useful tool in studying random walks, see
for example Karlsson-Ledrappier \cite{karlsson-ledrappier,
  karlsson-ledrappier2} and Bjorklund \cite{bjorklund}.  We now give a
brief description of this construction, giving full details in Section
\ref{section:horofunction}.
 
Let $X$ be a metric space, and $x_0$ a basepoint.  For each point $x$
in $X$, one defines the horofunction $\rho_x$ determined by $x$ to be
the function $\rho_x : X \to \mathbb{R}$
\[ \rho_x(z) := d_X(x, z) - d_X(x, x_0).  \]
This gives an embedding of $X$ in the space $C(X)$ of (Lipschitz-)
continuous functions on $X$, which we shall consider with the
compact-open topology (we emphasize that we use uniform convergence on
compact sets, not uniform convergence on bounded sets). With this
topology, the closure of $\rho(X)$ in $C(X)$ is compact,
even if $X$ is not locally compact; it is called the
\emph{horofunction compactification} of $X$ and denoted by $\Xh$. In
particular, there is a $\mu$-stationary measure $\nu$ on $\Xh$.

We now consider a basic but fundamental example in detail.

\begin{example}[$\R$] \label{ex:R} %
Consider $X = \R$, with the usual metric. In this case the
horofunction boundary $\Xh$ consists of $\rho(X)$ together
with precisely two additional functions, namely
$\rho_{\infty}(x) := -x$, and $\rho_{-\infty}(x) := x$.
\end{example}

This example turns out to be very important in our case; 
indeed, if $X$ is Gromov hyperbolic, then the
restriction of an arbitrary horofunction to a geodesic is equal (up to a bounded additive error)
to one of the horofunctions described above, i.e. $\rho_x$ or $\rho_{\pm
  \infty}$.

In Example \ref{ex:wedge}, the horofunction boundary equals the
Gromov boundary as a set, but the topology is different: namely, any sequence of
horofunctions $\seq{ \rho_{x_n} }$ corresponding to a sequence of
points $\seq{x_n}$ which leaves every compact set converges to the
horofunction $\rho_{x_0}$ associated to the basepoint $x_0$.

The Gromov boundary may be recovered from the horofunction boundary by
identifying functions which differ by a bounded amount, but in general
the horofunction boundary may be larger than the Gromov boundary, and
is not a quasi-isometry invariant of the space.  

\begin{example}
Consider $X = \Z \cross \Z / 2 \Z$, with the $L^1$-metric, $d_X((x,
i), (y, j)) = \norm{x-y} + \norm{i-j}$. Then the sequences $\rho_{n,
  0}$ and $\rho_{n, 1}$ have different values on $(0, 1)$, and so
converge to distinct horofunctions, and in fact in this case the
horofunction boundary consists of the product of the Gromov
boundary with $\Z / 2 \Z$.
\end{example}

We shall distinguish two different types of horofunctions. We say a
horofunction $h$ is \emph{finite} if $\inf_{x\in X} h(x) > - \infty$,
and is \emph{infinite} if $\inf h = - \infty$.  This partitions $\Xh$
into two subsets: we shall write $\xq$ for the set of finite
horofunctions, and $\xr$ for the set of infinite horofunctions.

\emph{The local minimum map.} %
We shall now construct a map from the horofunction boundary to the
Gromov boundary.  Recall that the restriction of a horofunction $h$ to
a geodesic $\gamma$ in $X$ is coarsely equal to one of the standard
horofunctions on $\R$: in particular, it has (coarsely) at most one
local minimum on $\gamma$.  Thus, if $h$ is bounded below on $\gamma$,
we can map $h$ to the location where it attains its minimum, getting a
map $\phi : \xq \to X$. On the other hand, if the horofunction is not
bounded below, then we can pick a sequence $(x_n)$ of points for which
the value of the horofunction tends to $-\infty$; it turns out that
such a sequence converges to a unique point in the Gromov boundary,
and the limit is independent of the choice of $(x_n)$.  Thus, we can
extend $\phi$ to a map $\phi : \Xh \to X \cup \partial X$.  We show
that this map is continuous on $\xr$ and $G$-equivariant, and that the
stationary measure $\nu$ is supported on the infinite horofunctions
$\xr$. Therefore the stationary probability measure $\nu$ on $\Xh$
restricts to a probability measure on $\xr$, and pushes forward to a
$\mu$-stationary probability measure $\tnu$ on $\Xb$.  We may then
complete the argument using with the geometric properties of the
action of $G$ on $\Xb$.

\emph{Plan of the paper.} %
In Section \ref{section:background}, we review some useful material
about Gromov hyperbolic spaces, and fix notation. In Section
\ref{section:horofunction}, we develop the properties of the
horofunction boundary that we will use, including the local minimum
map, and the behaviour of shadows. In these initial sections we give
complete proofs in the non-proper case of certain statements that are
already known in the proper case. In Section
\ref{section:convergence}, we use the horofunction boundary to show
that almost every sample path converges to the Gromov boundary.  In
Section \ref{section:applications}, we use the convergence to the
boundary result to show results on positive drift, sublinear tracking,
and the growth rate of translation distance, and then finally in
Section \ref{section:poisson} we show that if the action of $G$ is
acylindrical, and $\mu$ has finite entropy, then the Gromov boundary
with the hitting measure is the Poisson boundary for the random walk.

\subsection{Acknowledgements}

We would like to thank Jason Behrstock, Danny Calegari, Romain Dujardin, Camille
Horbez, Vadim Kaimanovich, Anders Karlsson, Andrei Malyutin and Samuel Taylor for
helpful conversations. 
The first author gratefully acknowledges the support of the Simons Foundation and PSC-CUNY.

\section{Background on $\delta$-hyperbolic spaces} \label{section:background}

Let $X$ be a Gromov hyperbolic space, i.e. a geodesic metric space
which satisfies the $\delta$-slim triangles condition. 
We will not assume that $X$ is \emph{proper}, i.e. that closed balls are compact, 
but we will always assume that it is \emph{separable}, i.e. that it contains a dense countable subset.
We shall write
$d_X$ for the metric on $X$, and $B_X(x, r)$ for the closed ball of
radius $r$ about the point $x$ in $X$. 
We shall now recall a few facts on the geometry of $X$.

\subsection{Notation}

We shall write $f(x) =  O(\delta)$ to mean that the absolute value of
the function $f$ is bounded by a number which only depends on
$\delta$, though this need not be a linear multiple of $\delta$.
Similarly, we shall write $A = B + O(\delta)$ to mean
that the difference between $A$ and $B$ is bounded by a
constant, which depends only on $\delta$.

\subsection{Coarse geometry}

Recall that the Gromov product in a metric space is defined to be
\[ \gp{x_0}{x}{y} := \tfrac{1}{2} ( d_X(x_0, x) + d_X(x_0, y) - d_X(x,
y) ).  \]
In a $\delta$-hyperbolic space, for all points $x_0, x$ and $y$, the
Gromov product $\gp{x_0}{x}{y}$ is equal to the distance from $x_0$ to
a geodesic from $x$ to $y$, up to an additive error of at most
$\delta$: if we write $[x,y]$ for a choice of geodesic from $x$ to
$y$, then
\begin{equation} \label{eq:gp estimate} %
d_X(x_0, [x, y]) = \gp{x_0}{x}{y} + O(\delta),  
\end{equation}
see e.g. \cite{bh}*{III.H 1.19}.
Moreover, for any three points $x, y, z \in X$ one has the following inequality
\begin{equation} 
\label{E:gpineq}
\gp{x_0}{x}{y} \geqslant \min \{ \gp{x_0}{x}{z}, \gp{x_0}{y}{z} \} - \delta,
\end{equation}
which we shall refer to as the triangle inequality for the Gromov
product.

We now recall the definition of the Gromov boundary of $X$, which
we shall write as $\Xb$.  We say that a sequence $\seq{x_n} \subseteq X$ is a \emph{Gromov sequence}
if $\gp{x_0}{x_m}{x_n}$ tends to infinity as $\min
\{ m, n \}$ tends to infinity. We say that two Gromov sequences $\seq{x_n}$
and $\seq{y_n}$ are equivalent if $\gp{x_0}{x_n}{y_n}$
tends to infinity as $n$ tends to infinity.
The \emph{Gromov boundary} $\partial X$ is defined as the set of equivalence classes of Gromov sequences.

We can extend the Gromov product to the boundary by
\[ \gp{x_0}{x}{y} = \sup \liminf_{m,n \to \infty}
\gp{x_0}{x_m}{y_n}, \]
where the supremum is taken over all sequences $(x_m)_{m \in \N} \to
x$ and $(y_n)_{n \in \N} \to y$. 
With this definition, the triangle inequality \eqref{E:gpineq}
also holds for any three points $x,y, z$ in $X \cup \Xb$,
but with a larger additive constant $O(\delta)$ instead of $\delta$ 
(see e.g. \cite{bh}*{III.H Remark 3.17(4)}).

The Gromov product on the boundary may be used to define a complete
metric on $\Xb$, see Bridson and Haefliger \cite{bh}*{III.H.3} for the
proper case, and V{\"a}is{\"a}l{\"a} \cite{vaisala} for the non-proper
case.  Moreover, the space $X \cup \partial X$ can be equipped with a
topology such that the relative topologies on both $X$ and $\Xb$ are
equal respectively to the usual topology on $X$, and the
above-mentioned metric topology on $\partial X$.

If $U \subset X$ we shall write $\overline{U}^\delta$ for the closure
of $U$ in $\Xg$.  If $X$ is proper, then $\Xb$ is
compact, but it need not be compact if $X$ is not proper. However, a
bounded set does not have limit points in the Gromov boundary,
i.e. for $B_X(x_0, r) = \{x \in X : d_X(x_0, x) \le r \}$ we have
$\overline{B_X(x_0, r)}^\delta = B_X(x_0, r)$.

\subsection{Quasigeodesics} \label{section:qg}

Let $I$ be a connected subset of $\R$, and let $X$ be a metric
space. A $(Q, c)$-quasigeodesic is a (not necessarily continuous) map
$\gamma \colon I \to X$ such that for all $s$ and $t$ in $I$,
\[ \frac{1}{Q}\norm{t - s} -c \le d_X(\gamma(s), \gamma(t)) \le Q
\norm{t - s} + c.  \]
If $I = \R$, then we will call the quasigeodesic $\gamma$ a
\emph{bi-infinite quasigeodesic}. If $I = [0, \infty)$, then we shall
call $\gamma$ a \emph{quasigeodesic ray based at $\gamma(0)$}. If the
metric space is $\delta$-hyperbolic, then quasigeodesics have the
following stability property, which is often referred to as the Morse
Lemma.

\begin{proposition} \label{prop:morse} %
Let $X$ be a $\delta$-hyperbolic space.  Given numbers $Q$ and $c$,
there is a number $L$ such that for any two points $x$ and $y$ in $X
\cup \partial X$, any two $(Q, c)$-quasigeodesics connecting $x$ and
$y$ are contained in $L$-neighbourhoods of each other.
\end{proposition}

We shall refer to a choice of constant $L$ in Proposition
\ref{prop:morse} above as a Morse constant for the quasi-geodesic
constants $Q$ and $c$. 

For any choice of basepoint $x_0$ and every point $x$ in the boundary
there is a quasigeodesic ray based at $x_0$ which converges to the
point $x$, and any two points in the boundary are connected by a
bi-infinite quasigeodesic. In fact, the quasigeodesics may be chosen
to have quasigeodesic constants $Q$ and $c$ bounded above in terms of
the hyperbolicity constant $\delta$, independently of the choice of
basepoint or boundary points, see e.g. Kapovich and Benakli
\cite{kapovich-benakli}.  By choosing $Q$ and $c$ sufficiently large,
we may assume that we have chosen these constants so that at least one
of the $(Q, c)$-quasigeodesics is continuous, see e.g. Bridson
and Haefliger \cite{bh}*{III.H.1}.

\subsection{Nearest point projection}

We will use the fact that in a $\delta$-hyperbolic space nearest point projection onto a geodesic
$\gamma$ is coarsely well defined, i.e.\ there is a constant $K_2$,
which only depends on $\delta$, such that if $p$ and $q$ are nearest
points on $\gamma$ to $x$, then $d_X(p,q) \le K_2$.

We will make use of the following \emph{reverse triangle inequality}.

\begin{proposition} \label{prop:reverse triangle} %
Let $\gamma$ be a geodesic in $X$, $y \in X$ a point, and $p$ a nearest point
projection of $y$ to $\gamma$. Then for any $z \in
\gamma$ we have 
\begin{equation} \label{eq:reverse triangle} %
 d_X(y, z) = d_X(y, p) + d_X(z, p) + O(\delta), 
\end{equation}
and furthermore, any geodesic from $z$ to $y$ passes within distance
$O(\delta)$ of $p$.
\end{proposition}

\begin{proof}
The upper bound for $d_X(y, z)$ is immediate from the usual triangle
inequality.  To prove the lower bound, by the definition of nearest
point projection,
\begin{align*}
d_X(y, p) & = d_X(y, [p, z]). \\
\intertext{Recall that by \eqref{eq:gp estimate},}
d_X(y, p) & \leqslant (p \cdot z)_y + \delta, \\
\intertext{and writing out the Gromov product, we get}
d_X(y, p) & = \frac{1}{2} \Big( d_X(y, p) + d_X(y, z) - d_X(p, z) \Big) +
\delta, 
\end{align*}
which yields 
$$d_X(y, z) \geqslant  d_X(y, p) + d_X(p, z) - 2 \delta, $$
as required. This implies that a path consisting of $[y,p] \cup [p,
z]$ is a $(1, 2 \delta)$-quasigeodesic, and so by stability of
quasigeodesics in a $\delta$-hyperbolic space, this path is contained
in an $O(\delta)$-neighbourhood of any geodesic $[y, z]$ from $y$ to
$z$, so in particular, the distance from $p$ to a $[y, z]$ is at most
$O(\delta)$.
\end{proof}

Finally we show that if two points $x$ and $y$ in $X$ have nearest
point projections $p_x$ and $p_y$ to a geodesic $\gamma$, and $p_x$
and $p_y$ are sufficiently far apart, then the path $[x, p_x] \cup
[p_x, p_y] \cup [p_y, y]$ is a quasigeodesic, and in fact has the same
length as a geodesic from $x$ to $y$, up to an additive error
depending only on $\delta$.

\begin{proposition} \label{prop:npp2}
Let $\gamma$ be a geodesic in $X$, and let $x$ and $y$ be two points
in $X$ with nearest points $p_x$ and $p_y$ respectively on
$\gamma$. Then if $d_X(p_x, p_y) \ge O(\delta)$, then
\[ d_X(x, y) = d_X(x, p_x) + d_X(p_x, p_y) + d_X(p_y, y) + O(\delta),  \]
and furthermore, any geodesic from $x$ to $y$ passes within an
$O(\delta)$-neighbourhood of both $p_x$ and $p_y$.
\end{proposition}

This is well known, see e.g. \cite{Maher_linear}*{Proposition
  3.4}.

Given a point $x \in X$ and a number $R > 0$ the \emph{shadow}
$S_{x_0}(x, R)$ is defined to be
\[ S_{x_0}(x, R) := \{ y \in X : \gp{x_0}{x}{y} \ge d_X(x_0, x) - R
\}.  \] 
There are a number of similar definitions in the literature, and we
emphasize that we define shadows to be subsets of $X$, rather than
subsets of say $X \cup \Xb$ or $\Xb$.  We will refer to the quantity
$d_X(x_0, x) - R$ as the \emph{distance parameter} of the shadow, and
it is equal to the distance from $x_0$ to the shadow, up to an
additive error depending only on $\delta$.

By the triangle inequality for the Gromov product \eqref{E:gpineq}, 
for any two points $y$ and $z$ in the closure of a shadow
$\overline{S_{x_0}(x, R)}^\delta$, there is a lower bound on their
Gromov product
\begin{equation} \label{eq:gp min} %
\gp{x_0}{y}{z} \ge d_X(x, x_0) - R + O(\delta).
\end{equation}

We now show that the nearest point projection of a shadow $S_{x}(y,
R)$ to a geodesic $[x, y]$ is contained in a bounded neighbourhood of
the intersection of the shadow with $[x, y]$, and the same result
holds for the complement of the shadow.

\begin{proposition} \label{prop:npp shadow}
Let $z$ be a point in the shadow $S_{x}(y, R)$, let $\gamma$ be a
geodesic from $x$ to $y$, and let $p$ be a nearest point to $z$ on
$\gamma$. Then 
\begin{equation} \label{eq:npp shadow} 
d_X(y, p) \le R + O(\delta). 
\end{equation}
If $z \not \in S_{x}(y, R)$, then 
\begin{equation} \label{eq:npp shadow complement} 
d_X(y, p) \ge R + O(\delta). 
\end{equation}
\end{proposition}

\begin{proof}
If $z$ lies in the shadow $S_{x}(y, R)$,
\[ \gp{x}{y}{z} \ge d_X(x, y) - R. \]
Using the definition of the Gromov product, we may rewrite this as
\[ d_X(x, z) - d_X(z, y) \ge d_X(x, y) - 2R, \]
and then using \eqref{eq:reverse triangle}, and the fact that $x, p$
and $y$ lie in that order on a common geodesic, gives
\[ R + O(\delta) \ge d_X(p, y), \]
as required. If $z$ does not lie in $S_x(y, R)$, then the same
argument works, with the opposite inequality.
\end{proof}

As a consequence, the complement of a shadow is almost a shadow:

\begin{corollary} \label{C:shadowcompl}
The complement of the shadow $S_x(y, R)$ is contained in the shadow $S_y(x, \widetilde{R})$, where $\widetilde{R} = d_X(x, y) - R + O(\delta)$.
\end{corollary}

\section{The horofunction boundary} \label{section:horofunction}

Let $(X, d_X)$ be a metric space.  A function $f : X \to \mathbb{R}$
is called \emph{$1$-Lipschitz} if for each $x, y \in X$ we have
$$|f(x) - f(y)| \leqslant d_X(x, y).$$
Clearly, $1$-Lipschitz functions are uniformly continuous. For each
$x_0 \in X$, let us define
$$\textup{Lip}^1_{x_0}(X) := \{ f \colon X \to \mathbb{R} \ : \
f \textup{ is 1-Lipschitz, and } f(x_0) = 0\}, $$
the space of $1$-Lipschitz functions which vanish at $x_0$.
We shall endow the space $\textup{Lip}^1_{x_0}(X)$ with the topology of pointwise convergence.
Note that, since all elements of $\textup{Lip}^1_{x_0}(X)$ are uniformly continuous with the same modulus of 
continuity, this topology is equivalent to the topology of 
uniform convergence on compact sets, which is also equivalent to the compact-open topology as $\mathbb{R}$ is a metric space.

\begin{proposition}
Let $X$ be a separable metric space. Then for each $x_0 \in X$, the space 
$\textup{Lip}^1_{x_0}(X)$
is compact, Hausdorff and second countable (hence metrizable).
\end{proposition}

\begin{proof}
Note that for any function $f \in \textup{Lip}^1_{x_0}(X)$
and each $z \in X$ we have
$$|f(z)| = |f(z) - f(x_0)| \leqslant d_X(x_0, z)$$
hence the space $\textup{Lip}^1_{x_0}(X)$ is a closed subspace of an infinite 
product of compact spaces, hence it is compact by Tychonoff's theorem.
Let $C(X)$ be the space of real-valued continuous functions on $X$,
with the compact-open topology.  
As $\R$ is Hausdorff, $C(X)$ is also Hausdorff, hence so is $\textup{Lip}^1_{x_0}(X)$. Since $X$ is a
separable metric space, it is second countable; moreover, as
$\mathbb{R}$ is also second countable, $C(X)$ is second countable and so is $\textup{Lip}^1_{x_0}(X)$.
\end{proof}

Let $x_0 \in X$ be a basepoint. We define the \emph{horofunction} map
$\rho$ to be
\begin{align*} 
\rho \colon X & \to C(X) \\
y & \mapsto \rho_y(z) := d_X(z, y) - d_X(x_0, y),
\end{align*}
where we write $\rho_y$ for $\rho(y)$.  Note that the horofunction map
$\rho$ depends on the choice of basepoint, but as we shall usually
consider horofunctions defined from some fixed basepoint we omit this
from the notation. In the few cases we consider horofunctions with
different basepoints we will change notation on an ad hoc basis. For
each $y$, the horofunction $\rho_y$ is $1$-Lipschitz, as
\begin{equation} \label{eq:bound}  
\norm{ \rho_y(z_1) - \rho_y(z_2) } = \norm{ d_X(z_1, y) - d_X(z_2, y) }
\le d_X(z_1, z_2),
\end{equation}
and moreover $\rho_y(x_0) = 0$ for all $y$, hence $\rho$ maps $X$ into 
$\textup{Lip}^1_{x_0}(X, \mathbb{R})$.

\begin{lemma}
The map $\rho \colon X \to C(X)$ defined above is continuous and injective.
\end{lemma}

\begin{proof}
For any $y \in X$, the function $\rho_y(z) = d_X(z, y) - d_X(x_0, y)$ achieves a
unique minimum at $z = y$, so the map $\rho$ is injective on $X$. 
The map $\rho$ is continuous, as for any $x, y, z \in X$ we have 
\begin{align*}
\norm{ \rho_x(z) - \rho_y(z) } & = \norm{ d_X(z, x) - d_X(x_0, x) - d_X(z, y) +
  d_X(x_0, y)  } \\
& \le \norm{  d_X(z, x) - d_X(z, y) } + \norm{  d_X(x_0, y) - d_X(x_0, x) }
\\ 
& \le 2 d_X(x, y) 
\end{align*}
and so if $(y_n)_{n \in \N} \to y$ then $(\rho_{y_n})_{n \in \N} \to
\rho_y$ uniformly on compact sets, in fact uniformly on all of $X$.
\end{proof}

Let us now define the fundamental object we are going to work with.

\begin{definition}
Let $X$ be a separable metric space with basepoint $x_0 \in X$. We
define the \emph{horofunction compactification} $\overline{X}^h$ to be
the closure of $\rho(X)$ in $\textup{Lip}^1_{x_0}(X)$.  We shall call
the set $\partial X^h := \Xh \setminus X$ the \emph{horofunction
  boundary} of $X$.  Elements of $\Xh$ will be called
\emph{horofunctions}.
\end{definition}

Note that in the proper case, the space $\Xh$ contains $X$ as an open,
dense set.  In the non-proper case, although the map $\rho$ is
injective on $X$, the image $\rho(X)$ need not be open in $\Xh$, so
although $\Xh$ is compact, it is not a compactification of $X$ in the
standard sense.

\begin{lemma}
Let $G$ be a group of isometries of $X$. Then the action
of $G$ on $X$ extends to a continuous action by homeomorphisms on
$\Xh$, defined as 
\begin{equation} \label{E:action}
g.h(z) := h(g^{-1} z) - h(g^{-1} x_0)
\end{equation}
for each $g \in G$ and $h \in \Xh$.
\end{lemma}

\begin{proof}
The action of $g \in G$ on $X$ translates into an action on $\rho(X)$, 
by defining $g.\rho_y := \rho_{gy}$ for each $g \in G$ and $y \in X$. Let us observe that 
\begin{align*}
g. \rho_{y}(z) & = \rho_{gy}(z) \\ 
& = d_X(g y, z) - d_X(gy, x_0) \\ 
&= d_X(y, g^{-1} z) - d_X(y, g^{-1} x_0) \\ 
& = \rho_y(g^{-1}z) - \rho_y(g^{-1} x_0), 
\end{align*}
thus we can define the action of $g \in G$ on each $h \in \Xh$ as in \eqref{E:action}.
It is immediate from the
definition that if $h_n \to h$ pointwise then $g.h_n \to g.h$, hence
$g$ acts continuously on $\Xh$, and so acts by homeomorphisms, as $G$
is a group.
\end{proof}

We shall write $\overline{U}^h$ for the closure of $U$ in $\Xh$. We
remark that as $B(x_0, r)$ need not be compact, a sequence of points
contained in a bounded set may have images under $\rho$ which converge
to the horofunction boundary, i.e.  $\overline{\rho(B(x_0, r))}^h$ may
contain points in the horofunction boundary $\partial X^h$.

\subsection{Horofunctions in $\delta$-hyperbolic spaces} \label{S:horohyp}

We now record some basic observations about the behaviour of
horofunctions. We start by describing the restriction of a
horofunction $\rho_y$ to a geodesic $\gamma$ in $X$. As we shall see, 
for any geodesic
$\gamma$, the restriction of $\rho_y$ to $\gamma$ has a coarsely well
defined local minimum a bounded distance away from the nearest point
projection $p$ of $y$ to $\gamma$.  Moreover, for any point $z \in \gamma$, the
value of $\rho_y(z)$ is equal to $\rho_y(p) + d_X(p, z)$, up to bounded
error depending only on $\delta$. We now make this precise.

\begin{proposition}
Let $\gamma$ be a geodesic in $X$, $y \in X$, and let $p$ be a nearest point
projection of $y$ to $\gamma$. Then the restriction of $\rho_y$ to
$\gamma$ is given by
\begin{equation} \label{eq:rhogeodesic} 
\rho_y(z) = \rho_y(p) + d_X(z, p) + O(\delta), 
\end{equation}
for all $z \in \gamma$.
\end{proposition}

\begin{proof}
This follows from the reverse triangle inequality, Proposition
\ref{prop:reverse triangle}, by adding $- d_X(y, x_0)$ to both sides.
\end{proof}

We now describe the restriction of an arbitrary horofunction $h$ to a
geodesic $\gamma$ in $X$.
An \emph{orientation} for a geodesic $\gamma$ is a
strict total order on the points of $\gamma$, induced by a
choice of unit speed parameterization (thus, each geodesic has exactly two orientations). 
We may then define the \emph{signed distance function} along $\gamma$ to be
\[ d^+_\gamma(x, y) = \left\{ 
\begin{array}{rl} 
d_X(x, y) & \text{ if } x \le y \\ 
-d_X(x, y) & \text{ if } x \ge y. 
\end{array} \right. \]

\begin{proposition} \label{prop:horofunction geodesic} %
Let $h$ be a horofunction in $\Xh$, and let $\gamma$ be a geodesic in
$X$. Then there is a point $p$ on $\gamma$
such that the restriction of $h$ to $\gamma$
is equal to exactly one of the following two functions, up to bounded
additive error:
%
\begin{enumerate}
\item[]  either 
\begin{equation} \label{eq;sign1}
h(x) = h(p) + d_X(p, x) + O(\delta), \qquad x \in \gamma; 
\end{equation}
\item[]
or 
\begin{equation}
h(x) = h(p) + d^+_\gamma(p, x) + O(\delta), \qquad x \in \gamma \label{eq:sign}
\end{equation}
for some choice of orientation on $\gamma$.
\end{enumerate}
\end{proposition}

So for example, for geodesic rays starting at the basepoint $x_0$, the
graphs of these functions are equal to one of the two graphs shown
below, up to an error of $O(\delta)$.

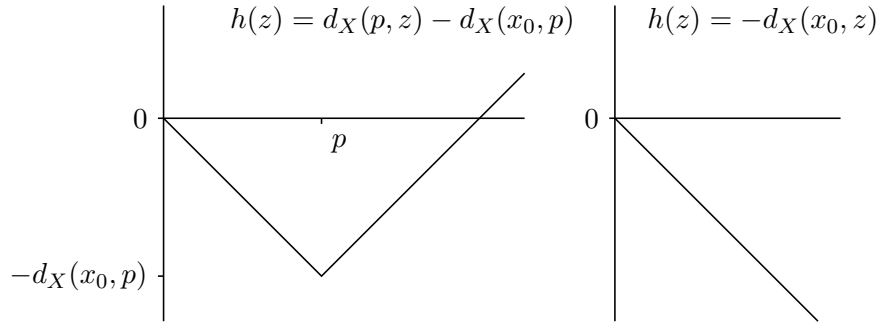
\begin{figure}[H]
 \begin{center}
\begin{tikzpicture}[scale=0.3]

\draw[semithick] (12,-13) -- (12,-13.25) node [below right] {$p$};
\draw[semithick] (4.75,-13) node [left] {$0$} -- (21,-13);
\draw[semithick] (5,-13) -- (12,-20) -- (21,-11);
\draw[semithick] (25,-8) -- (25,-22);
\draw[semithick] (24.75,-13) node [left] {$0$} -- (35,-13);
\draw[semithick] (5,-8) -- (5,-22);
\draw[semithick] (5,-20) -- (4.75,-20) node [left] {$-d_X(x_0, p)$};
\draw[semithick] (25,-13) -- (34,-22);
\path (7.5,-9) node[anchor=base west] {$h(z) = d_X(p, z) - d_X(x_0, p)$};
\path (12.5,-12.5);
\path (26,-9) node[anchor=base west] {$h(z) = - d_X(x_0, z)$};

\end{tikzpicture}
\end{center} 
\caption{The behaviour of $h(z)$ along a geodesic ray starting at
  $x_0$.}
\label{pic:graph}
\end{figure}

\begin{proof}
Let $h$ be a horofunction in $\Xh$, let $\seq{ \rho_{y_n} }$ be
sequence of horofunctions which converge to $h$, and let $p_n$ be the
nearest point projection of $y_n$ to $\gamma$.

First, suppose that there is a subsequence of the $\seq{\rho_{y_n} }$
for which the projections $p_n$ to $\gamma$ are bounded,
i.e. contained in a subinterval $I$ of $\gamma$ of finite length. We
may pass to a further subsequence, which by abuse of notation we shall
also call $\seq{ \rho_{y_n} }$, such that the projections $p_n$
converge to a point $p \in \gamma$, and in fact are all within
distance $1$ of $p$. Therefore, for this subsequence, $\rho_{y_n}(x) =
\rho_{y_n}(p) + d_X(p, x) + O(\delta)$, by \eqref{eq:rhogeodesic}. As
$\rho_{y_n} \to
h$ pointwise, this implies that $h(x) = h(p) + d_X(p, x) + O(\delta)$ for each $x \in \gamma$,
as required.

Now consider the case in which the nearest point projections $p_n$
eventually exit every compact subinterval of $\gamma$. In this case it
will be convenient to choose $p$ to be a nearest point on $\gamma$ to
the basepoint $x_0$, and we may pass to a subsequence such that $p_n >
p$ for all $n$, for some choice of orientation on $\gamma$.

For any $x \in \gamma$ all but finitely many $p_n$ satisfy $p_n >
x$. As $p_n$ is a nearest point projection of $y_n$ to $\gamma$, and
$p$ is a nearest point projection of $x_0$ to $\gamma$, we may use the
reverse triangle inequality to rewrite $\rho_{y_n}(x)$ in terms of
$d_X(x_0, p)$, $d_X(y_n, p_n)$, and distances between points on the
geodesic $\gamma$. There are two cases, depending on whether $x > p$,
or $x \le p$, illustrated below in Figure \ref{pic:horofunction
  geodesic}.

The sign on each line segment in Figure \ref{pic:horofunction
  geodesic} indicates the sign of the corresponding line segment in
the approximation for $h_{y_n}(x)$. This shows that 
\[ \rho_{y_n}(x) = 
\left\{ 
\begin{array}{ll}
-d_X(x_0, p) - d_X(p, x) + O(\delta), \text{ if } x > p, \\
-d_X(x_0, p) + d_X(p, x) + O(\delta), \text{ if } x\le p.
\end{array}
\right.   \]
As $\seq{ \rho_{y_n} }$ converges to $h$, this implies the same coarse
equalities for $h(x)$. As $h(p) = -d_X(x_0, p) + O(\delta)$, and using
the definition of signed distance along $\gamma$, this gives
\eqref{eq:sign} above.
\end{proof}

\begin{figure}[H] 
\begin{center}
\begin{tikzpicture}[scale=0.4]

\tikzstyle{point}=[circle, draw, fill=black, inner sep=0pt, minimum width=2.5pt]

\draw (0,0) -- (10,0) node [below] {$\gamma$};
\draw (2,0) -- node [midway, right] {$-$} (2,4);
\draw (8,0) -- node [midway, left] {$+$} 
                            node [midway, right] {$-$} (8,-4);

\draw (2,0) -- node [midway, above] {$-$} (5,0);
\draw (5,0) -- node [midway, above] {$-$} 
               node [midway, below] {$+$} (8,0);

\filldraw (2,4) node [point, label=right:$x_0$] {};
\filldraw (2,0) node [point, label=below:$p$] {};
\filldraw (8,0) node [point, label=above:$p_n$] {};
\filldraw (8,-4) node [point, label=right:$y_n$] {};
\filldraw (5,0) node [point, label=below:$x$] {};

\begin{scope}[xshift=+18cm]

\draw (-4,0) -- (10,0) node [below] {$\gamma$};
\draw (2,0) -- (2,4);
\draw (8,0) -- (8,-4);

\filldraw[black] (2,4) node [point, label=right:$x_0$] {};
\filldraw[black] (2,0) node [point, label=below:$p$] {};
\filldraw[black] (8,0) node [point, label=above:$p_n$] {};
\filldraw[black] (8,-4) node [point, label=right:$y_n$] {};
\filldraw[black] (-2,0) node [point, label=below:$x$] {};

\draw (2,0) -- node [midway, right] {$-$} (2, 4);
\draw (-2,0) -- node [midway, below] {$+$} (2,0)
             -- node [midway, below] {$+$} 
                node [midway, above] {$-$} (8,0)
             -- node [midway, left] {$+$} 
                node [midway, right] {$-$} (8,-4);

\end{scope}

\end{tikzpicture}
\end{center} 
\caption{The horofunction $\rho_{y_n}$ restricted to a geodesic $\gamma$.}
 \label{pic:horofunction geodesic}
\end{figure}
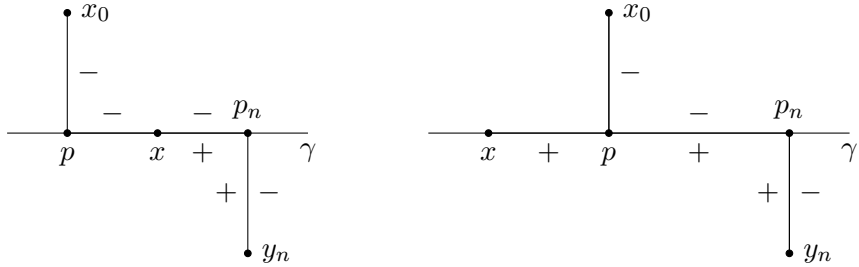

We say a function $f(x)$ has no coarse local maxima if $f(x) = g(x) +
O(\delta)$, where $g(x)$ has no local maxima. Similarly we say $f(x)$
has at most one coarse local minima if $f(x) = g(x) + O(\delta)$,
where $g(x)$ has at most one local minima. So we have shown that any
horofunction $h$ restricted to a geodesic $\gamma$ has no coarse local
maxima on $\gamma$, and at most one coarse local minimum on $\gamma$.

\subsection{A partition of the horofunction boundary}

For any horofunction $h \in \Xh$ we may consider
\[ \inf(h) = \inf_{y \in X} h(y), \] 
which takes values in $[ -\infty, 0]$. We may partition $\Xh$ into two
sets depending on whether or not $\inf(h) = - \infty$. 

\begin{definition}
We shall denote $\Xh_F$ the set of \emph{finite horofunctions} 
\[ \Xh_F := \{ h \in \Xh : \inf(h) >  - \infty \},  \]
and $\Xh_\infty$ the set of \emph{infinite horofunctions}
\[ \Xh_\infty := \{ h \in \Xh : \inf(h) = - \infty \}. \]
\end{definition}

Clearly, both $\Xh_F$ and $\Xh_\infty$ are invariant for the action of
$G$.  Note moreover that if a horofunction $h$ is contained in
$\rho(X)$, i.e $h = \rho_y$ for some $y \in X$, then
$$\inf(\rho_y) = - d_X(x_0, y) > - \infty, $$
and the infimum is achieved at the unique point $y \in X$, hence
$\rho(X)$ is contained in the set $\xq$ of finite horofunctions.  More
generally, by the same proof, if $B$ is a bounded subset of $X$, then
$$\overline{\rho(B)}^h \subset \xq.$$

Note however that the subset $\xr \subset \Xh$ need not be compact, so
there may be sequences of elements of $\xr$ which do not have
subsequences which converge in $\xr$.  If a sequence of horofunctions
$\seq{h_n}$ converges to $h$, this does not in general imply that
$\inf h_n$ converges to $\inf h$. For example, in Example
\ref{ex:wedge}, consider the sequence $\seq{ \rho_{x_n} }$, where
$x_n$ is the point distance $n$ from $0$ in the $n$-th ray. Then $\inf
\rho_{x_n} = -n$, but $\seq{ \rho_{x_n} }$ converges to $\rho_{x_0}$,
for which $\inf \rho_{x_0} = 0$. In fact, in this example there are
sequences in $\xr$ which converge to horofunctions in $\xq$. If we set
$h_n$ to be equal to the limit of $( \rho_{x_k} )_{k \in \N}$, where
the $x_k$ are points distance $k$ from $0$ along the $n$-th ray, then
$h_n \in \xr$, but $\seq{ h_n }$ converges to $\rho_{x_0}$.

\begin{lemma} \label{L:horoineq}
For each basepoint $x_0 \in X$, each horofunction $h \in \Xh$ and each pair of points $x, y \in X$
the following inequality holds:
\begin{equation} \label{E:hgp}
\min\{ - h(x), - h(y) \} \leqslant \gp{x_0}{x}{y} + O(\delta).
\end{equation}
\end{lemma}

\begin{proof}
Let $z \in X$. Then one has, by the triangle inequality 
\begin{align*} \gp{x_0}{x}{z} & = \frac{d_X(x_0, x) + d_X(x_0, z) -
  d_X(x, z)}{2},  \\
\intertext{which implies}
\gp{x_0}{x}{z} & \geqslant d_X(x_0, z) - d_X(x, z),
\intertext{and by definition, the right hand side is equal to
  $-\rho_z(x)$, which gives}
\gp{x_0}{x}{z} & \ge -\rho_z(x).
\end{align*}
Now, by $\delta$-hyperbolicity, combined with the previous estimate, one has 
\begin{align*}
\gp{x_0}{x}{y} & \geqslant \min\{ \gp{x_0}{x}{z}, \gp{x_0}{y}{z} \} -
\delta, \\
\intertext{which we may rewrite as}
\gp{x_0}{x}{y} & \geqslant \min\{ -\rho_z(x), -\rho_z(y) \} - \delta.
\end{align*}
Since every horofunction is the pointwise limit of functions of type $\rho_z$, the claim follows.
\end{proof}

\begin{definition}
We define a sequence $\seq{x_n} \subseteq X$ to be \emph{minimizing} for a horofunction $h$ if $h(x_n) \to -\infty$ as $n \to \infty$.
\end{definition}

We shall now prove that every minimizing sequence is a Gromov sequence, hence it has a limit in the Gromov boundary.

\begin{lemma} \label{L:converges}
Let $h \in \xr$ an infinite horofunction, and $(y_n)$ a sequence of points in $X$ such that 
$h(y_n) \to -\infty$. Then the sequence $(y_n)$ converges to a point in the Gromov boundary 
of $X$. Moreover, two minimizing sequences for the same horofunction converge to the same point in 
the Gromov boundary. 
\end{lemma}

\begin{proof}
By Lemma \ref{L:horoineq} one has 
$$ \gp{x_0}{y_n}{y_m} \geqslant \min\{-h(y_n), -h(y_m)\} - O(\delta) \to \infty$$
as $\min\{m,n\} \to \infty$, proving the first claim.

To prove uniqueness of the limit, suppose that there are two sequences $(x_n)$ and $(y_n)$ such that $h(x_n) \to -\infty$ and $h(y_n) \to -\infty$.
Then, by Lemma \ref{L:horoineq}, the Gromov product $\gp{x_0}{x_n}{y_n} \to \infty$, hence by definition the sequences $(x_n)$ and $(y_n)$ 
converge to the same point in the Gromov boundary $\Xb$.
\end{proof}

Note finally that if $h \in \xr$, then there is actually a
quasigeodesic sequence $(y_n)$ in $X$ such that $h(y_n) \to -
\infty$. Furthermore, we can recover the Gromov boundary from the
horofunction boundary as follows. Define an equivalence relation on
$\Xh$ by $h_1 \sim h_2$ if $\sup_{x \in X} \norm{h_1(x) - h_2(x)}$ is
finite. This collapses $\xq$ to a single point, and the equivalence
classes in $\xr$ are precisely the point pre-images of the local
minimum map $\phi \colon \xr \to \Xb$, so the Gromov boundary is
homeomorphic to $\xr /_\sim$. However, we will not use this result so
we omit the proof.

\subsection{The local minimum map}

We now define a map $\phi \colon \Xh \to \Xg$, which may be thought of
as the ``local minimum'' map, which sends a horofunction $h$ to the
location at which it attains its minimum. If the horofunction $h$ does
not attain a minimum in $X$, it turns out that it makes sense to
think of the minimum value as lying in the Gromov boundary. We now make
this precise.

\begin{definition} 
The \emph{local minimum} map $\phi \colon \Xh \to \Xg$ is defined as follows.

\begin{itemize}

\item If $h \in \xq$, i.e. $\inf(h) > - \infty$, then define 
$$\phi(h) := \{ x \in X \ : \ h(x) \leqslant \inf h + 1 \}$$
the set of points of $X$ where the value of $h$ is close to its infimum; 

\item if $h \in \xr$, i.e. $\inf(h) = - \infty$, then choose a sequence
  $\seq{y_n}$ with $h(y_n) \to - \infty$, and set 
  $$\phi(h) := \lim_{n \to \infty} y_n$$ 
 to be the limit point of $(y_n)$ in the Gromov boundary.

\end{itemize}

\end{definition}

\begin{lemma} The local minimum map $\phi : \Xh \to X \cup \Xb$ is well-defined and $G$-equivariant.
\end{lemma}

\begin{proof}
By Lemma \ref{L:converges}, the map is well-defined on $\xr$: indeed, every minimizing sequence
for $h$ converges in the Gromov boundary, and any two minimizing sequences yield the same limit.

To prove equivariance, let us first pick $h \in \xq$, and
$x \in \phi(h)$. Then for each $y \in X$ we have $h(x) \leqslant h(y) + 1$,
thus 
$$g.h(g x) = h(x) - h(g^{-1}x_0) \leqslant h(y) - h(g^{-1}x_0) + 1 = g.h(gy) + 1$$
for each $y \in X$, hence the value of $g.h$ at $gx$ is close to its infimum hence $gx \in \phi(g.h)$.
If instead $h \in \xr$, then let $(y_n)$ a minimizing sequence for $h$. Then by definition of the action one gets
$$g.h(g y_n) = h(y_n) - h(g^{-1}x_0) \to -\infty$$
hence $(g y_n)$ is a minimizing sequence for $g.h$, so $\phi(g.h) = g.\phi(h)$ as required.
\end{proof}

\begin{lemma} \label{lemma:bdd diameter}
There exists $K$, which depends only on $\delta$, such that for each finite horofunction $h \in \xq$ we have 
$$\textup{diam }\phi(h) \le K.$$
\end{lemma}

\begin{proof}
Let $x, y \in \phi(h)$, for some $h \in \xq$, and consider the restriction of $h$ 
along a geodesic segment from $x$ to $y$.
By Proposition \ref{prop:horofunction geodesic}, the restriction has at most one
coarse local minimum: hence, since $x$ and $y$ are coarse local minima of $h$,
the distance between $x$ and $y$ is universally bounded in terms of $\delta$.
\end{proof}

\begin{proposition} \label{prop:phi cts on R}
Let $\seq{h_n}$ a sequence of horofunctions which converges to some infinite horofunction 
$h \in \xr$. Then $\phi(h_n)$ converges to $\phi(h)$ in the Gromov boundary.
As a consequence, the local minimum map $\phi: \xr \to \Xb$ is continuous.
\end{proposition}

(If $h_n$ is a finite horofunction, in the above statement we mean that  $x_n \to \phi(h)$ for any choice 
 of $x_n \in \phi(h_n)$.)

\begin{proof}
Let $N > 0$, and let $(h_n)$ a sequence of horofunctions which converge to $h \in \xr$.
Let us pick a minimizing sequence $(x_m)_{m \in \mathbb{N}}$ for $h$, and for each $n$ 
a sequence $(y_{m, n})_{m \in \mathbb{N}}$, such that $y_{m, n} \to \inf h_n \in \R \cup \{-\infty\}$ as $m \to \infty$, 
so that $\phi(h) = [x_m]$ and $\phi(h_n) = [y_{m,n}]$. The goal is to prove that $\gp{x_0}{\phi(h_n)}{\phi(h)} \to \infty$ as $n \to \infty$.

Since $(x_m)$ is a Gromov sequence and it is minimizing for $h$, there exists $m_0$ such that
$$h(x_{m_0}) \leqslant -N -1$$
and for each $m, m' \geqslant m_0$ one has 
$$\gp{x_0}{x_m}{x_{m'}} \geqslant N+1.$$

Since $h_n \to h$ pointwise, there exists $n_0$ such that 
$$h_n(x_{m_0}) \leqslant -N$$
for each $n \geqslant n_0$. Now, since $(y_{m,n})$ is minimizing for $h_n$, there exists $m_1 = m_1(N, n)$ such that $m_1 \geqslant m_0$
and
$$h_n(y_{m, n}) \leqslant -N \qquad \textup{for each }m \geqslant m_1$$
Hence, by Lemma \ref{L:horoineq} we have 
$$\gp{x_0}{x_{m_0}}{y_{m, n}} \geqslant \min \{ - h_n(x_{m_0}), - h_n(y_{m, n}) \} \geqslant N$$
and by property \eqref{E:gpineq}, for all $m, m' \geqslant m_1$
$$\gp{x_0}{x_{m'}}{y_{m, n}} \geqslant \min \{ \gp{x_0}{x_{m_0}}{y_{m, n}}, \gp{x_0}{x_{m_0}}{x_{m'}} \} - \delta \geqslant N - \delta$$
thus 
$\gp{x_0}{\phi(h)}{\phi(h_n)} = \sup \liminf_{m, m'} \gp{x_0}{x_{m'}}{y_{m,n}} \geqslant N -\delta$ for $n \geqslant n_0$, as claimed.
\end{proof}

\begin{corollary} \label{cor:surjective}
The local minimum map $\phi \colon \xr \to \Xb$ is surjective.
\end{corollary}

\begin{proof}
Pick $\lambda \in \partial X$. By construction, there exists a sequence $\seq{\lambda_n}$ of points of $X$ 
which converge to $\lambda$. By compactness, the sequence of horofunctions $\rho_{\lambda_n}$
has a subsequence $(\rho_{\lambda_{n_k}})_{k \in \N}$ which converges to some $h \in \Xh$; since 
$\inf \rho_{\lambda_n} \le \rho_{\lambda_n}(\lambda_n) = -d_X(x_0, \lambda_n) \to -\infty$, 
we have that $h$ belongs to $\xr$; thus, by the Proposition, $\lambda_{n_k} \to \phi(h)$, hence 
by uniqueness of the limit $\phi(h) = \lambda$, as required. 
\end{proof}

However, we emphasize that if $(h_n)$
converges to $h$ in $\xq$, then $(\phi(h_n))$ need not converge to
$\phi(h)$. 
For instance, in the countable wedge of rays of Example \ref{ex:wedge}, if $x_n$ is the point on the branch $X_n$ 
at distance $n$ from the base point $x_0$, then $\rho_{x_n} \to \rho_{x_0}$ in the horofunction 
compactification, but the sequence $\seq{x_n}$ does not converge in the Gromov boundary.

\subsection{Horofunctions and shadows}

We define the \emph{depth} of the shadow $S = S_{x_0}(x, R)$
to be the quantity
$$dep(S) := 2R - d_X(x_0, x).$$
We now show that we may characterize points in a shadow in terms of
the value of the corresponding horofunction at the basepoint and the
depth.

\begin{lemma} \label{L:shadowhoro}
If $S = S_{x_0}(x, R)$ is a shadow, then 
$y \in S$ 
if and only if 
$$\rho_y(x) \le dep(S).$$
\end{lemma}

\begin{proof}
By definition of shadow one has
\[ \gp{x_0}{x}{y} \ge d_X(x_0, x) - R  \]
hence by writing out the Gromov product 
\[ \tfrac{1}{2}( d_X(x_0, x) + d_X(x_0, y) - d_X(x, y)  ) \ge d_X(x_0, x) - R  \]
and by simplifying we get
\[ -\rho_y(x)  \ge d_X(x_0, x) - 2R  \]
which proves the claim.
\end{proof}

\begin{corollary}
For any shadow $S = S_{x_0}(x, R)$, the closure of $S$ in
$\Xh$ is
\[\overline{S}^h = \{ h \in \Xh : h(x) \le dep(S) \}, \]
where $dep(S) = 2R - d_X(x_0, x)$.
\end{corollary}

It will also be useful to know how shadows are related to the topology
of the Gromov boundary $\Xb$. We shall use the following property of
shadows: there is a constant $R_0$, which only depends on the action of
$G$ on $X$, such that for any $g \in G$ the closure of the shadow $S_{x_0}(g x_0,
R_0)$ in $\partial X$ contains a non-empty open set. This follows from the Proposition below.

\begin{proposition} \label{prop:shadow open}
Let $G$ be a group acting by isometries on a separable Gromov
hyperbolic space $X$, such that $G$ contains at least one hyperbolic
isometry. Then there is a number $R_0 > 0$
such that for any $g \in G$ the set $\overline{S_{x_0}(g x_0,
R_0)}^\delta$ contains a limit point of $G x_0$ in its interior.
\end{proposition}

This follows from the following result from Blach\`ere, Ha\"issinsky
and Mathieu \cite{bhm}.

\begin{proposition} \cite{bhm}*{Proposition 2.1} \label{prop:bhm} %
For any  $\e > 0$ sufficiently small, and any $A > 0$,
there are positive numbers $C$ and $R_0$, such that for any $R > R_0$,
and any $x \in X, y \in \Xb$ with $\gp{x}{x_0}{y} \le A$,
\[ B_\e(y, \tfrac{1}{C}e^{\e (R - d_X(x_0, x))}) \subset \overline{
  S_{x_0}(x, R) }^\delta \cap \Xb \subset B_\e(y, Ce^{\e (R -d_X(x_0,
  x))}), \]
where $B_\e(y, r)$ is the ball of radius $r$ about $y$ in the
metric $d_\e$ on $\Xb$.
\end{proposition}

\begin{proof}[Proof (of Proposition \ref{prop:shadow open})]
Since $G$ contains at least one hyperbolic isometry, then 
the limit set $\overline{Gx_0}^\delta$ contains at least two points
in $\partial X$, which we call $\alpha^+$ and $\alpha^-$. 
Let $\alpha$ be a quasigeodesic from $\alpha^+$ to $\alpha^-$, 
and let $p$ be a
closest point on $\alpha$ to the basepoint $x_0$. Given a group
element $g \in G$, consider the translate $g \alpha$, as illustrated
in Figure \ref{pic:quasi-axis} below.

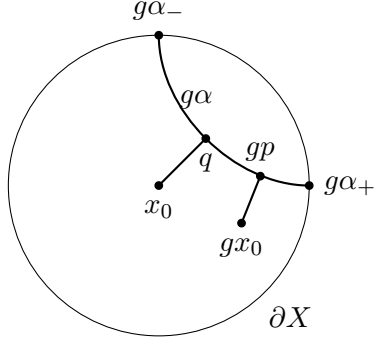
\begin{figure}[H]
\begin{center}
\begin{tikzpicture}[scale=0.5]

\tikzstyle{point}=[circle, draw, fill=black, inner sep=0pt,
minimum width=2.5pt]

\draw (0,0) circle (4);

\draw [thick] (4, 0) node [point, label=right:$g \alpha_+$] {}
..controls (2, 0) and (0, 2) .. (0, 4) node [point, label=above:$g
\alpha_-$] {};

\draw (1, 2.3) node {$g \alpha$};

\draw (3.5, -3.5) node {$\Xb$};

\draw [thick] (0, 0) node [point, label=below:$x_0$] {} -- (1.25, 1.25) node
[point, label=below:$q$] {};

\draw [thick] (2.2, -1) node [point, label=below:$g x_0$] {} -- (2.7, 0.25) node
[point, label=above:$g p$] {};

\end{tikzpicture}
\end{center} 
\caption{The translate of the quasigeodesic $\alpha$ under $g$.}
\label{pic:quasi-axis}
\end{figure}

For any element $g \in G$, let $q$ be a nearest point on $g \alpha$ to
$x_0$.  Any quasigeodesic from $x_0$ to either $g \alpha_+$ or $g
\alpha_-$ passes within distance $O(\delta)$ of $q$, so 
$g x_0$ lies within distance $d_X(x_0, \alpha) + O(\delta)$ of at
least one of these quasigeodesics, which we may assume has endpoint
$\alpha_+$, up to relabeling.

Therefore, the product $\gp{g x_0}{x_0}{g \alpha^+}$ is bounded above independently of $g \in G$, 
hence by Proposition \ref{prop:bhm}, there is a number $R_0$,
(depending only on $\delta$ and the choice of $\alpha$) 
such that the closure $\overline{S_{x_0}(g
x_0, R)}^\delta \cap \partial X$ contains an open set containing $g \alpha_+ \in \overline{Gx_0}^\delta$, for any $R \ge
R_0$, as required.
\end{proof}

\subsection{Horofunctions and weak convexity of shadows}

\emph{A priori}, 
shadows need not be convex, or even
quasi-convex. However, we now show various results about horofunctions
and nested shadows which we can think of as weak versions of
convexity. For example, for any two points contained in a shadow
$S_{x_0}(x, R)$, the geodesic connecting them is contained in
$S_{x_0}(x, R + O(\delta))$. We start by showing that the value
of a horofunction along a geodesic is bounded by its values on the
endpoints, up to an additive error of $O(\delta)$.

\begin{lemma} \label{L:qc}
Let $X$ a $\delta$-hyperbolic, geodesic metric space. Then there exists a constant $C$, which depends only on $\delta$, such that given any geodesic segment 
$[z_1, z_2]$ in $X$, with $z_1, z_2 \in X$, the following holds: 
\begin{equation} \label{E:qc}
\min \{ \rho_{z_1}(x), \rho_{z_2}(x) \} - C \leqslant \rho_y(x) \leqslant \max \{ \rho_{z_1}(x), \rho_{z_2}(x) \} + C
\end{equation}
for any $y \in [z_1, z_2]$ and any $x \in X$.
\end{lemma}

\begin{proof}
Let us first assume that $x, x_0$ belong to $[z_1, z_2]$. Up to swapping $z_1$ and $z_2$, we can assume $x \in [z_1, x_0]$; then we have the bound
$$\rho_{z_1}(x) = - d_X(x, x_0) \leqslant \rho_y(x) = d_X(x, y) - d_X(x_0, y) \leqslant d_X(x, x_0) = \rho_{z_2}(x)$$
which yields the claim.
Now, in the general case let $p_x$ be the closest point projection of
$x$ to $[z_1, z_2]$, and $p_{x_0}$ the projection of $x_0$. By the
definition of $\rho_z$,
\begin{align*}
\rho_z(x) & = d_X(x, z) - d_X(x_0, z). \\
\intertext{Then by the reverse triangle inequality (Proposition \ref{prop:reverse triangle}) we have for each $z \in X$,}
\rho_z(x) & = d_X(x, p_x) + d_X(p_x, z) - d_X(x_0, p_{x_0}) -
d_X(p_{x_0}, z) + O(\delta),
\end{align*}
hence for each $i = 1, 2$
$$\rho_{y}(x) - \rho_{z_i}(x) = \widetilde{\rho}_{y}(p_x) - \widetilde{\rho}_{z_i}(p_x) + O(\delta),$$
where $\widetilde{\rho}_{z}(x) := d_X(x, z) - d_X(p_{x_0}, z)$ denotes the horofunction based at $p_{x_0}$. Since now $p_{x_0}$ and $p_x$ lie on 
$[z_1, z_2]$, the claim follows by the previous case.
\end{proof}

Lemma \ref{L:qc} implies the following weak convexity property of both
shadows and their complements.

\begin{corollary} \label{cor:weak convexity} %
There exists a constant $C$, which depends only on $\delta$, such that 
for each shadow $S = S_{x_0}(x, R)$ the following hold:
\begin{enumerate}
\item if $z_1$ and $z_2$ belong to $S$, 
then the geodesic segment $[z_1, z_2]$ lies in $S_{x_0}(x, R + C)$;
\item
if $z_1$ and $z_2$ do not belong to $S$, then the geodesic segment $[z_1, z_2]$ does not intersect $S_{x_0}(x, R-C)$.
\end{enumerate}
Moreover, for each $(Q, c)$ which satisfy Proposition \ref{prop:morse}, the above statements still hold with ``geodesic" replaced by ``$(Q,c)$-quasi-geodesic", where this time $C$ depends on $\delta, Q$, and $c$.
\end{corollary}

\begin{proof}
If $z_1, z_2$ belong to $S$, then $\rho_{z_i}(x) \leqslant dep(S)$ for each $i = 1, 2$ by Lemma \ref{L:shadowhoro}, hence Lemma \ref{L:qc} implies
$$\rho_y(x) \leqslant \max\{ \rho_{z_1}(x), \rho_{z_2}(x) \} + C \leqslant dep(S) + C = dep(S')$$
 where $S' = S_{x_0}(x, R+C/2)$, 
thus $y \in S'$ once again by Lemma \ref{L:shadowhoro}. The proof of (2) is similar, using the left-hand side of equation \eqref{E:qc}.
The extension to quasi-geodesic is immediate by the fellow-traveling property of Proposition \ref{prop:morse}.
\end{proof}

\section{Convergence to the boundary} \label{section:convergence}

In this section we prove the following theorem.

\begin{theorem} \label{T:conv} %
Let $G$ be a countable group of isometries of a
geodesic, separable, $\delta$-hyperbolic space $X$ (not necessarily
proper), and let $\mu$ be a non-elementary probability measure on $G$. 
Then for each $x_0 \in X$, almost every sample path $(w_n x_0)_{n \in \N}$
converges in $X \cup \partial X$ to a point of the Gromov boundary
$\partial X$.
\end{theorem}

The proof of the theorem takes several steps, and it exploits the
action of $G$ on the space of probability measures both on the
horofunction compactification $\Xh$ and on the Gromov boundary
$\partial X$.  
The strategy of the proof is the following:

\begin{enumerate}
\item
By compactness, there is a stationary measure $\nu$ on $\Xh$ (Lemma \ref{L:existstat}).
\item
The measure $\nu$ does not charge the finite part of the boundary: $\nu(\xr) = 1$ (Proposition \ref{prop:qzero}).  
\item 
By the martingale convergence theorem (Proposition \ref{P:martingale}), for almost every sample path the sequence of measures $\seq{w_n \nu}$
converges to some measure $\nu_\omega$ in $\mathcal{P}(\Xh)$.
\item
By pushing the sequence forward to $\partial X$, using the local
minimum map $\phi$, almost every sequence $\seq{w_n \tnu}$ converges to some measure $\phi_* \nu_\omega$ in $\mathcal{P}(\partial X)$ (Lemma \ref{lemma:mart2}).
\item
Almost every sample path $\seq{w_n x_0}$ has a subsequence which converges to a point $\lambda$ in the Gromov boundary $\partial X$ (Proposition \ref{P:convsub}).
\item 
Thus, by Lemma \ref{lemma:delta}, almost every sample path $\seq{w_n}$ has a subsequence $(w_{n_k})_{k
  \in \N}$ such that $(w_{n_k} \tnu)_{k \in \N}$ converges to a delta-measure $\delta_\lambda$, 
for some $\lambda \in \Xb$.
\item 
Since the limit exists, almost every sequence $\seq{w_n \tnu}$ converges to a delta-measure $\delta_\lambda$ (Proposition \ref{prop:delta}).
\item
We prove that the fact that $\seq{w_n \tnu}$ converges to $\delta_\lambda$ implies that the sample path $\seq{w_n x_0 }$
converges to $\lambda \in \partial X$ (Proposition \ref{P:topconv}).
\end{enumerate}
In the rest of the section we shall work out the details of the proof.

We remark that sample paths do not in general converge to points in
the horofunction compactification.

\begin{example}
Consider the nearest neighbour random walk on the Cayley graph of $F_2
\cross \Z / 2\Z$, with respect to the standard generating set $\langle
a, b, c \mid [a, c], [b, c], c^2 \rangle$, and with basepoint $x_0$
corresponding to the identity element. If $g \in F_2$, then
$\rho_{g}(c) = 1$, and $\rho_{gc}(c) = -1$. As almost every sample
path $w_n$ hits each coset of $F_2$ infinitely often, sample paths do
not converge in $\Xh$, almost surely.
\end{example}

\subsection{Random walks}

We briefly review some background material on random walks and fix
notation.  Let $\mu$ be a probability distribution on $G$; 
the \emph{step space} of the random walk generated by
$\mu$ is the measure space $(G^\Z, \mu^\Z)$, which is the countable
infinite product of the measure spaces $(G, \mu)$.
Each element of $G^\mathbb{Z}$ is a sequence $\seqz{g_n}$, 
whose entries are the increments of our (bi-infinite) random walk. 
The shift map $T \colon G^\Z \to G^\Z$ sends $\seqz{g_n}$ to
$\seqz{g_{n-1}}$, and is measure preserving and ergodic.

We define the \emph{location} of the random walk at time $n$, which we
shall denote $w_n$ to be
\[ w_n = \left\{
\begin{array}{ll}
g_{0}^{-1} g_{-1}^{-1} \ldots g_{n+1}^{-1} & \text{ if } n \le -1 \\
1 & \text{ if } n = 0 \\
g_1 g_2 \ldots g_n & \text{ if } n \ge 1.
\end{array}
\right. \]
This gives a map $G^\Z \to G^\Z$, defined by $(g_n)_{n \in \Z} \mapsto
(w_n)_{n \in Z}$. 
We shall denote the range of the location map as $\Omega$, to distinguish it from the step space, 
and call $\P$ the pushforward to $\Omega$ of the product measure $\mu^\Z$ on the step space $G^\Z$.
We shall refer to $(\Omega, \P)$ as the \emph{path space}, and elements
$\w \in \Omega$ as \emph{sample paths} of the random walk. The shift
map $T$ acts on $\Omega$ by $T^k \colon (w_n)_{n \in \Z} \mapsto
\seqz{w_k^{-1} w_{n+k}}$, and is measure preserving and
ergodic.

\subsection{Stationary measures}

Let $M$ be a metrizable topological space, and denote
$\mathcal{P}(M)$ the space of Borel probability measures on $M$.
The space $\mathcal{P}(M)$ is endowed with the weak-* topology,
which is defined by saying $\nu_n \to \nu$ if for each continuous
bounded function $f$ on $M$ one has $\nu_n(f) \to \nu(f)$.

If now $G$ is a countable group which acts on $M$ by homeomorphisms, we
denote $g\nu$ the pushforward of $\nu \in \mathcal{P}(M)$ under the
action of $g \in G$, i.e. $g \nu(U) = \nu(g^{-1} U)$, and define the
\emph{convolution operator} $\star \colon \mathcal{P}(G) \times
\mathcal{P}(M) \to \mathcal{P}(M)$ as the average of the
pushforwards:
$$\mu \star \nu := \sum_{g \in G} \mu(g) \ g\nu.$$ 
We say that 
a probability distribution $\nu$ on $M$ is
\emph{$\mu$-stationary} if $\mu \star \nu = \nu$, i.e. for each Borel set $U$ we have 
\begin{equation} \label{eq:mustationary} %
\nu(U) = \sum_{g \in G} \mu(g) \nu(g^{-1}U).   
\end{equation}
A space $M$ equipped with a $\mu$-stationary measure $\nu$ is called a \emph{$(G, \mu)$-space}.
Now, if we fix a base point $x_0 \in M$, we shall write $\tmu \in \mathcal{P}(M)$ for 
the pushforward of $\mu$ under the orbit
map, i.e. if $U \subset M$ then $\tmu(U) = \mu(\{ g \in G : g
  x_0 \in U\})$. We shall write $\mun{n}$ for the $n$-fold
convolution of $\mu$ with itself on $G$, we shall write $\tmun{n}$ for
the pushforward of $\mun{n}$ to $M$, and finally, we shall write
$\overline{ \mu }_n$ for the Ces\`aro averages of the pushforward
measures, $\overline{ \mu }_n := \tfrac{1}{n}(\tmu + \tmu_2 + \cdots +
\tmu_n)$.  Classical compactness arguments yield the following:

\begin{lemma} \label{L:existstat} %
Let $G$ be a countable group which acts by homeomorphisms on a compact metric space $M$, 
and let $\mu$ be a probability
distribution on $G$. Then there exists a $\mu$-stationary Borel
probability measure $\nu$ on $M$.
\end{lemma}

\begin{proof}
Since $M$ is a compact metrizable space, then $\mathcal{P}(M)$ is compact in the
weak-$*$ topology. Then any weak-$*$ limit point of the sequence $\seq{\overline{\mu}_n}$ of the 
Ces\`aro averages is $\mu$-stationary.
An alternate proof follows from the Schauder-Tychonoff fixed point theorem.
\end{proof}

Applying the above arguments to $\Xh$ implies that there exists a $\mu$-stationary measure
$\nu$ on $\Xh$, i.e. $(\Xh, \nu)$ is a
$(G, \mu)$-space.  We now show that the measure $\nu$ is supported on
$\xr$.

\begin{proposition} \label{prop:qzero} %
Let $G$ be a non-elementary countable group of isometries of a
separable Gromov hyperbolic space $X$. Let $\mu$ be a non-elementary 
probability distribution on $G$, and let $\nu$ be a $\mu$-stationary measure
on $\Xh$. Then
\[ \nu(\xq) = 0. \]
\end{proposition}

In order to show that some set $Y$ has measure zero, the basic idea is
to consider the translates $gY$ of the set $Y$, and to consider the
supremum of the measures of these sets. If we choose a translate $gY$
with measure very close to the supremum, then by $\mu$-stationarity,
if $h g Y$ is another translate with $\mu(h) > 0$, then $\nu(h g Y)$
will also be close to the supremum. If there are enough disjoint
translates with $\nu$-measures close to the supremum, then the total
measure of $\nu$ is strictly greater than one, which contradicts the
fact that $\nu$ is a probability measure. We now make this precise.

\begin{lemma} \label{lemma:measure zero} %
Let $G$ a countable group acting by homeomorphisms on a metric space $M$,
$\mu$ a probability distribution on $G$ whose support generates $G$ as a semigroup, 
and $\nu$ a $\mu$-stationary probability measure on $M$. 
Moreover, let us suppose that $Y \subset M$ has the property that there is a
sequence of positive numbers $\seq{ \e_n }$ such that for any
translate $f Y$ of $Y$ there is a sequence $\seq{ g_n }$ of group
elements (which may depend on $f$), such that the translates $fY,
g_1^{-1} fY, g_2^{-1} fY, \ldots$ are all disjoint, and for each
$g_n$, there is an $m \in \N$, such that $\mu_m(g_n) \ge \e_n$. Then
$\nu(Y) = 0$.
\end{lemma}

The proof of this is a variation on \cite{Maher_heegaard}*{Lemma 3.5},
but we provide a proof for the convenience of the reader.

\begin{proof}
Suppose that $s := \sup \{ \nu(f Y) : f \in G\} > 0$. Choose $N >
2/s$, let $\e = \min \{ \e_i : 1 \le i \le N \}$, and let $\e_s =
\e/N$. Finally, choose $f$ such that the harmonic measure of $f Y$ is
within $\e_s$ of the supremum, i.e. $\nu(f Y) \ge s - \e_s$. By
hypothesis, there is a sequence of group elements $g_1, \ldots, g_N$
such that the $N$ translates $g_1^{-1} f Y, \ldots, g_N^{-1} f Y$, are
all disjoint, and for each $g_n$ there is an $m$ such that with
$\mu_m(g_n) \ge \e$.

The harmonic measure $\nu$ is $\mu$-stationary, and hence
$\mun{m}$-stationary for any $m$, which implies
\begin{align*} 
\nu(f Y) &= \sum_{h \in G} \mun{m}(h) \nu(h^{-1} f Y).
\intertext{%
For any element $g \in G$ we may rewrite this as
}
\mun{m}(g) \nu(g^{-1} f Y) &= \nu(f Y) - \sum_{h \in G \setminus g}
\mun{m}(h) \nu(h^{-1} f Y). \\
\intertext{%
As we have chosen $f Y$ to have measure within $\e_s$ of the supremum,
this implies 
}
\mun{m}(g) \nu(g^{-1} f Y) & \ge s-\e_s -\sum_{h \in G \setminus g}
\mun{m}(h) \nu(h^{-1} f Y).
\intertext{%
The harmonic measure of each translate of $Y$ is at most the supremum
$s$, 
}
\mun{m}(g) \nu(g^{-1} f Y) & \geqslant s - \e_s - s \sum_{h \in G
  \setminus g} \mun{m}(h),
\intertext{%
and the sum of $\mun{m}(h)$ over all $h \in G \setminus g$ is equal to
$1-\mun{m}(g)$, which implies that 
}
\mun{m}(g) \nu(g^{-1} f Y) & \geqslant s - \e_s - s(1 - \mun{m}(g) ).
\intertext{%
As the semi-group support $\langle \supp(\mu) \rangle_+$ is equal to
$G$, for any element $g \in G$ there is an $m$ such that $\mun{m}(g) >
0$. For such an $m$ we may divide by $\mun{m}(g)$ to give the
following estimate for the harmonic measure of the translate,
}
\nu(g^{-1} f Y) & \geqslant s - \e_s /\mun{m}(g).
\end{align*}
In particular, such an estimate holds for each of the $N$ disjoint
translates $g_i^{-1} f Y$, and furthermore, $\mun{m}(g_i) \ge \e$ for each
$i$. This implies that
\[ 
\nu(\bigcup_{i = 1}^N g_i^{-1} f Y) = \nu(g_1^{-1} f Y) + \dots + \nu(g_N^{-1} f Y)  \geqslant N(s - \e_s/\e). \]
As we chose $N > 2/s$, and $\e_s/\e = 1/N$, this implies that the total
measure $\nu(\bigcup g_i^{-1} f Y)$ is greater than one, a contradiction.
\end{proof}

We now complete the proof of Proposition \ref{prop:qzero}.  Recall
that the translation length $\tau(g)$ of an isometry $g$ of $X$ is defined to be
\[ \tau(g) := \lim_{n \to \infty} \tfrac{1}{n} d_X(x_0, g^n x_0). \]
This definition is independent of the base point $x_0$, and $\tau(g) >
0$ if and only if $g$ is a hyperbolic isometry, and furthermore
$\tau(g^k) = k\tau(g)$.

\begin{proof}[Proof of Proposition \ref{prop:qzero}]
We shall apply Lemma \ref{lemma:measure zero} taking as $Y$ the set of
horofunctions whose local minimum lies in a given ball around the
base-point: precisely, $Y = \{ h \in \xq \ : \ \phi(h) \cap B(x_0, r)
\neq \emptyset \}$, where $\phi$ is the local minimum map, and $B(x_0,
r)$ is a ball of radius $r$ in $X$.  As $G$ is non-elementary it
contains hyperbolic isometries of arbitrarily large translation
length. Choose a hyperbolic isometry $g$ with translation length
$\tau(g)$ greater than $2r + K$, where $K$ is the bound on the
diameter of $\phi(h)$ from Lemma \ref{lemma:bdd diameter}.  Now, the
translates $g^{-n} f B(x_0, r)$ are all at least distance $\tau(g) -
2r > K$ apart, hence no $\phi(h)$ can intersect two of them, so the
sets $g^{-n} f Y$ are all disjoint.  As the semi-group generated by
the support of $\mu$ is equal to $G$, for each $n$
there is an $m$ such that $\mun{m}(g^{n}) > 0$. Set $\e_n =
\mun{m}(g^n)$, for some such $m$, then Lemma \ref{lemma:measure zero}
implies that $\nu(Y) = 0$. As this holds for every
$r$, this implies that $\nu(\xq) = 0$, as required.
\end{proof}

The measure $\nu$ is therefore supported on $\xr$, and as $\phi$ is
continuous on $\xr$, the measure $\nu$ pushes forward to a Borel
probability measure 
$$\tnu := \phi_* \nu$$ 
on the Gromov boundary $\Xb$.  The measure $\tnu$ is a
$\mu$-stationary probability measure on $\Xb$, so $(\Xb, \tnu)$ is a
$(G, \mu)$-space.  We now show that $\tnu$ is non-atomic, which
implies that $\nu$ is non-atomic as well. Recall that if the action of
$G$ on $X$ is non-elementary, then $G$ does not preserve any finite
subset of the boundary $\Xb$.

\begin{lemma} \label{lemma:non-atomic} %
Let $G$ be a countable group which acts by isometries on a separable
Gromov hyperbolic space $X$. Let $\mu$ be a non-elementary probability
distribution on $G$, and let $\nu$ be a $\mu$-stationary measure on
$\Xh$, with pushforward $\tnu$ on $\Xb$ under the local minimum map
$\phi$.  Then the measure $\tnu$ is non-atomic (hence so is
$\nu$). Furthermore, any $\mu$-stationary measure on $\Xb$ is the
pushforward of a $\mu$-stationary measure on $\xr$.
\end{lemma}

\begin{proof}
We first observe that if there are atoms, then there must be an atom
of maximal weight, as an infinite sequence of atoms $\seq{b_n}$ of
increasing weights has total measure greater than one. Let $m$ be the
maximal weight of any atom, and let $A_m$ be the collection of atoms
of weight $m$, which is a finite set. As $\tnu$ is
$\mu$-stationary, if $b \in A_m$, then
\[ \tnu(b) = \sum_{g \in G} \mu(g) \tnu(g^{-1}b).  \]
As no atom has weight greater than $m$, all elements of the orbit of $b$ under
$G$ must have the same weight $m$, so $A_m$ is a finite $G$-invariant
set, which contradicts the fact that $G$ is non-elementary.

Finally, as the local minimum map $\phi \colon \xr \to \Xb$ is
surjective, the pushforward map $\phi_* \colon \mathcal{P}(\xr) \to
\mathcal{P}(\Xb)$ is also surjective, see e.g. 
\cite{ab}*{Theorem 15.14}. 
If $\lambda$ is a $\mu$-stationary
measure on $\Xb$, then $\phi_*^{-1}(\lambda)$ is a non-empty, convex subspace of
the space of measures on $\xr$, which can also be seen as a subspace of the space
$\mathcal{P}(\Xh)$ of probability measures on $\Xh$.
Thus, the closure $H$ of $\phi_*^{-1}(\lambda)$ in $\mathcal{P}(\Xh)$ is 
compact, convex, and invariant under convolution
with $\mu$, so by the Schauder-Tychonoff fixed point theorem there is
a $\mu$-stationary measure $\nu$ in $H \subseteq \mathcal{P}(\Xh)$.
However, by Proposition \ref{prop:qzero}, $\nu$ vanishes on the set of 
finite horofunctions, hence it belongs to $\mathcal{P}(\xr)$, and 
since $\nu$ is a limit of elements in $\phi_*^{-1}(\lambda)$
and $\phi_*$ is continuous, we also 
have $\phi_* \nu = \lambda$, as required.
\end{proof}

\subsection{Convergent subsequences}

The goal of this section is to prove the following step in the proof of Theorem \ref{T:conv}:

\begin{proposition} \label{P:convsub} %
Let $G$ be a countable group of isometries of a
separable Gromov hyperbolic space $X$, and let $\mu$ be a non-elementary 
probability distribution on $G$.

Then, for $\P$-almost every sample path $\seq{w_n}$ there
is a subsequence of $\seq{\rho_{w_n x_0}}$ which converges to a
horofunction in $\xr$.  

As a corollary, $\P$-almost every sample path
$\seq{w_n}$ has a subsequence $(w_{n_k})_{k \in \N}$ such that
$(w_{n_k} x_0)_{k \in \N}$ converges to a point in the Gromov boundary
$\Xb$.
\end{proposition}

Given a shadow $S = S_{x_0}(x, R)$, we define the \emph{open shadow}
$S^\circ = S^\circ_{x_0}(x, R)$ to be the subset of $\Xh$ given by
\[S^\circ_{x_0}(x, R) := \{ h \in \Xh : h(x) < dep(S) \}. \]
As $\{x\}$ is compact and $(-\infty, dep(S))$ is open in $\R$, the set $S^\circ$
is an open subset of $\Xh$ contained in the interior of $\overline{S}^h$.

\begin{lemma} \label{L:cover}
For each $T <  0$, the set 
$\xr$ is contained in a countable collection of open shadows 
of depth $\le T$.
\end{lemma}

\begin{proof}
We have immediately from the definition
$$\xr \subseteq \{ h \in \Xh \ | \ \inf h < T \} = \bigcup_{x \in X} \{ h \in \Xh \ | \ h(x) < T \}.$$
Now, by picking a countable dense subset $\{x_i\}_{i \in \mathbb{N}}$ of $X$
and an enumeration $\{T_j\}_{j \in \mathbb{N}}$ of $(-\infty,T) \cap  \mathbb{Q}$, we have 
$$\xr \subseteq \bigcup_{i, j \in \mathbb{N}} \{ h \ | \ h(x_i) < T_j
\},$$
as required.
\end{proof}

We shall now define a \emph{descending shadow sequence} to be a sequence  
$\mathcal{S} = (\mathcal{S}_M)_{M \in \N}$, where 
each  $\mathcal{S}_M$ is a finite collection of shadows, 
and each shadow $S \in \mathcal{S}_M$ has depth $dep(S) \le -M$. 

Given  a descending shadow sequence $\mathcal{S}$, we shall introduce
(for convenience of notation) an indexing 
of all its shadows, i.e. $\bigcup_M \mathcal{S}_M = \{S_1, S_2, \dots\}$.
We say a $M$-tuple $I = (i_1, \dots, i_M)$ of positive integers is 
\emph{an index set of depth $M$} for $\mathcal{S}$ if each $j = 1, \dots, M$,
the shadow $S_{i_j}$ is an element of
$\mathcal{S}_j$ (hence, it has depth $\le -j$).  Given a descending shadow sequence $\mathcal{S}$,
and an index set $I = (i_1, \dots, i_M)$, we define the \emph{cylinder} $C_I$ to be the
intersection of the open shadows $S^\circ_{i_j}$ corresponding to the
indexed shadows $S_{i_j}$, i.e.
\[ C_I := S_{i_1}^\circ \cap \dots \cap S_{i_M}^\circ, \]
so $C_I$ is an open set in $\Xh$.  Given a number $M \in \N$, let
$\Sigma_M$ be 
\[ \Sigma_M := \bigcup_{\stackrel{I \text{ index set}}{ \text{of depth }M}} C_I = \bigcup S_{i_1}^\circ \cap \dots \cap
S_{i_M}^\circ, \]
where the union is taken over all index sets of depth $M$. The sets
$\Sigma_M$ form a nested sequence of open sets in $\Xh$, i.e.
\[  \Sigma_1 \supseteq \Sigma_2 \supseteq \Sigma_3 \supseteq \cdots.  \]
Finally, we observe that by Lemma \ref{L:shadowhoro}, if $h \in
\Sigma_M$ then $\inf h \le -M$, so
$$\bigcap_{M \in \N} \Sigma_M \subseteq \xr.$$

\begin{lemma} \label{L:conv} %
Let $\seq{y_n}$ be a sequence of points of $X$, and let
$(\mathcal{S}_M)_{M \in \mathbb{N}}$ be a descending shadow sequence.
If $\seq{\rho_{y_n}}$ intersects $\Sigma_M$ for each $M$, then there
exists a subsequence $(y_{n_k})_{k \in \N}$ such that
$(\rho_{y_{n_k}})_{k \in \N}$ converges to a horofunction in $\xr$.
\end{lemma}

\begin{proof}
Suppose the sequence $\seq{\rho_{y_n}}$ intersects each $\Sigma_M$. So
for each $M \in \N$ there is an $n_M$ such that $\rho_{y_{n_M}} \in
\Sigma_M$. As $\bigcap \Sigma_M \subset \xr$, each $\rho_{y_n}$ may
lie in only finitely many $\Sigma_M$, and so $n_M \to \infty$ as $M
\to \infty$. 
The horofunction $\rho_{y_{n_k}}$ lies in $\Sigma_{k}$, which is a union
of cylinders, so there is an index set $I_k = (i_1, \dots, i_{k})$
of depth $k$ such that
\[ \rho_{y_{n_k}} \in C_{I_k} = S^\circ_{i_1} \cap \cdots \cap
S^\circ_{i_{k}}. \]

There are only finitely many choices for the first entry $i_1$ in each
index set $I_k$, so we may pass to a further subsequence in which
${i_1}$ is constant. Choose $n_{1}$ to be the first element of this
subsequence, and relabel the remaining elements as $(n_k)_{k \in \N}$
for $k \ge 2$. Again, as there are only finitely many choices for the
second entry $i_2$ in the index set $I_k$, we may pass to a further
subsequence in which the indices $i_2$ are constant for all $k \ge
2$. Proceeding by induction, we may construct a subsequence
$(y_{n_k})_{k \in \N}$, and a sequence of indices $(i_k)_{k \in \N}$,
such that
\[ \rho_{y_{n_k}} \in S^\circ_{i_1} \cap \cdots \cap S^\circ_{i_{k}},  \]
for each $k$.

By compactness, the sequence $(\rho_{y_{n_k}})_{k \in \N}$ has a limit
point $h \in \Xh$.  Then by Lemma \ref{L:shadowhoro}, for each $j$ and
each $k \ge j$, if $S^\circ_{i_j} = S^\circ_{x_0}(x_j, R_j)$, we have $y_{n_k} \in S^\circ_{i_j}$, hence
$$\rho_{y_{n_k}}(x_j) \le -j, $$
thus by passing to the limit as $k \to \infty$ we have
$$h(x_j) \le -j \qquad \textup{for all }j \ge 1,$$
which gives $h \in \xr$. 
\end{proof}

\begin{lemma} \label{L:largeshadow} %
Let $\epsilon > 0$, and let $\nu$ be a $\mu$-stationary measure on
$\Xh$. Then there exists a finite descending shadow sequence
$\mathcal{S} = (\mathcal{S}_M)_{M \in \mathbb{N}}$ such that for each
$M$,
$$\nu(\Sigma_M) \ge 1-\epsilon.$$
\end{lemma}

\begin{proof}
Recall by Lemma \ref{L:cover} that for any $M \in \N$ there exists a
countable collection of shadows of depth $\le -M$ which covers $\xr$,
and $\xr$ has full measure by Proposition \ref{prop:qzero}.  Thus,
since probabilities are countably additive one can find a finite set
$\mathcal{S}_M := \{S_{M, 1}, \dots, S_{M, r_M}\}$ of shadows of depth
$\le -M$ such that the union $\overline{\mathcal{S}}_M := \bigcup_{i =
  1}^{r_M} S_{M, i}$ satisfies
$$\nu( \overline{\mathcal{S}}_M ) \ge 1-   2^{-M} \epsilon.$$
We may now set $\mathcal{S}$ to be the sequence $(\mathcal{S}_M)_{M
  \in \N}$, which is a descending shadow sequence.  We now
observe that $\nu(\Sigma_M) \geq 1 - \epsilon/2 - \dots - \epsilon/
2^{M} \ge 1-\epsilon$ as required.
\end{proof}

We may now complete the proof of Proposition \ref{P:convsub}.

\begin{proof}[Proof (of Proposition \ref{P:convsub})]
We shall show that the set $Z$ of sequences $\seq{w_n}$ in the path
space $(\Omega, \P)$ for which $\seq{\rho_{w_n}}$ does not have limit
points in $\xr$ has measure at most $\epsilon$, for each $\epsilon >
0$. Fix $\epsilon > 0$, and let $(\Sigma_M)_{M \in \mathbb{N}}$ be a
descending shadow sequence constructed according to
Lemma \ref{L:largeshadow}, using a measure $\nu$ which is a
$\mu$-stationary weak limit of the Ces\`aro averages $\overline{\mu}_n
= \tfrac{1}{n}(\tmun{1} + \cdots + \tmun{n})$.  Now, suppose that the
sequence $\seq{\rho_{w_n}}$ does not have limit points in $\xr$: then
by Lemma \ref{L:conv} there exists an index $M$ such that $\rho_{w_n}$
does not belong to $\Sigma_M$ for any $n$.  Thus we have the inclusion
$$Z \subseteq \bigcup_{M} \bigcap_{n} \{ (w_k)_{k \in \N} \ : \ \rho_{w_n}
\notin \Sigma_M \}, $$
so if we set $Y_M := \Xh \setminus \Sigma_M$, then
$$\mathbb{P}(Z) \le \sup_{M} \inf_{n} \tmun{n}(Y_M).$$
Then, by definition of the Ces\`aro averages, $\inf
\overline{\mu}_n(Y_M) \ge \inf \tmun{n}(Y_M)$, so this implies that
$$\mathbb{P}(Z) \le \sup_{M} \inf_{n} \overline{\mu}_n(Y_M).$$
Furthermore as $\nu$ is the weak limit points of the 
$\overline{\mu}_n$ of the Ces\`aro averages, and $Y_M$ is closed, we have for each $M$,
$$\inf_n \overline{\mu}_n(Y_M) \le \nu(Y_M), $$
hence by Lemma \ref{L:largeshadow}
$$\mathbb{P}(Z) \le \sup_M \nu(Y_M) \le \epsilon, $$
and the claim is proven. The corollary follows from Proposition \ref{prop:phi cts on R}.
\end{proof}

\subsection{The boundary action} \label{S:bdryaction}

We now prove Theorem \ref{theorem:convergence}, convergence to the
boundary. 
We start by showing that the action of $G$ on $\Xg$ satisfies the
following property (which need not hold for the action of $G$ on
$\Xh$): if the sequence $(g_n x_0)_{n \in \N}$ converges to a point in $\Xb$, then
the sequence $(g_n y)_{n \in \N}$ converges to the same point, for any $y \in X$.

\begin{lemma} \label{lemma:insideconv}
Let $X$ be a Gromov hyperbolic space.  If $x_0 \in X$ and the sequence
$\seq{g_n x_0}$ converges to a point $\lambda$ in $\Xb$, then the
sequence $\seq{g_n y}$ converges to the same point $\lambda$, for any
$y \in X$.
\end{lemma}

\begin{proof}
Consider the Gromov product
\[ \gp{x_0}{g_n x_0}{g_n y} = \tfrac{1}{2}( d_X(x_0, g_n x_0) +
d_X(x_0, g_n y) - d_X(g_n x_0, g_n y) ). \]
By the triangle inequality $d_X(x_0, g_n y) \ge d_X(x_0, g_n x_0) - d_X( g_n
x_0, g_n y)$, and as $g_n$ is an isometry, $d_X(g_n x_0, g_n y) = d_X(x_0,
y)$. This implies
\[ \gp{x_0}{g_n x_0}{g_n y} \ge d_X(x_0, g_n x_0) - d_X(x_0, y) , \]
which tends to infinity as $n$ tends to infinity, so $\seq{g_n y}$ converges
to the same limit point as $\seq{g_n x_0}$.
\end{proof}

The following is a version of Kaimanovich \cite{kaimanovich}*{Lemma
  2.2} in the non-proper case.

\begin{lemma} \label{lemma:action} %
Let $G$ be a group of isometries of a Gromov hyperbolic space $X$.
Let $\seq{g_n}$ be a sequence in $G$ such that $\seq{g_n x_0} \to
\lambda \in \Xb$. Then there is a subsequence $(g_{n_k})_{k \in \N}$
such that $(g_{n_k} x)_{k \in \N} \to \lambda$ for all but at most one
point of $X \cup \Xb$.
\end{lemma}

\begin{proof}
We show that there is a subsequence $(g_{n_k})_{k \in \N}$ for which
there is at most one point $b \in \Xb$ such that $(g_{n_k} b)_{k \in
  \N} \not \to \lambda$. Suppose there is a point $b_1$ in $\Xb$ such
that $(g_n b_1)_{n \in \N} \not \to \lambda$. This implies there is an
open set $U_1$ containing $\lambda$ such that infinitely many $g_n
b_1$ do not lie in $U_1$. Therefore, we may pass to a subsequence
$(g_{n_k})_{k \in \N}$ such that $g_{n_k} b_1 \not \in U_1$ for all
$k$. Now suppose there is another point $b_2 \in \Xb$ such that
$(g_{n_k} b_2)_{k \in \N} \not \to \lambda$. This implies there is an
open set $U_2$ containing $\lambda$ such that infinitely many $g_{n_k}
b_2$ do not lie in $U_2$. As before, we may pass to a subsequence,
which by abuse of notation we shall continue to call $(g_{n_k} )_{k
  \in N}$, such that $(g_{n_k} b_2)_{k \in \N} \not \in U_2$ for all
$n$.
As $U_1 \cap U_2$ is also an open neighbourhood of $\lambda$, 
it contains a shadow set of
the form $S_1 = \overline{ S_{x_0}(x, R) }^\delta$, with $\lambda$ contained
in the interior of the slightly smaller shadow 
$S_2 = \overline{ S_{x_0}(x, R-C) }^\delta$,  where $C$ will be chosen as the 
 weak convexity constant of Corollary \ref{cor:weak convexity} for $(Q,c)$-quasi-geodesics.
Let now $\gamma$ be a $(Q, c)$-quasigeodesic from $b_1$ to
$b_2$, and $y \in X$ a point on $\gamma$.
By Corollary \ref{cor:weak convexity}, since the endpoints $g_{n_k}b_1$ 
and $g_{n_k}b_2$ do not belong to $S_1$, then the point $g_{n_k}y$ 
does not belong to $S_2$.
However, this is a contradiction, as by Lemma \ref{lemma:insideconv} we know that 
$g_{n_k}y \to \lambda$, hence it must be eventually lie inside $S_2$.
\end{proof}

\subsection{Convergence of measures} \label{S:measconv}

We will use the following result, which goes back to Furstenberg
\cite{Furstenberg}*{Corollary 3.1}, see also Margulis
\cite{margulis}*{Chapter VI}.

\begin{proposition} \label{P:martingale}
Let $M$ be a compact metric space on which the countable group $G$
acts continuously, and $\nu$ a $\mu$-stationary Borel probability
measure on $M$.  Then for $\P$-almost all sequences $\seq{w_n}$ the limit
\[ \nu_\w := \lim_{n \to \infty} g_1 g_2 \ldots g_n \nu \] exists in
the space $\mathcal{P}(M)$ of probability measures on $M$.
\end{proposition}

To give a brief overview of the argument, one proves that, since the
measure is stationary, for each continuous function $f \in C(M)$ the
process
$$X_n :=  \int_X f(g_1 \dots g_n x) \ d\nu(x)$$
is a bounded martingale, hence converges almost surely; this defines a
positive linear functional on the space $C(M)$, which is thus
represented by a Borel measure.

Furthermore, if $w_n \nu \to \nu_\omega$ a.s., then we get the integral formula
\begin{equation} \label{eq:integral} 
\nu = \int_\Omega \nu_\w \ d \P(\w).   
\end{equation}

\begin{lemma} \label{lemma:mart2} %
Let $G$ be a countable group of isometries of a
separable Gromov hyperbolic space $X$, let $\mu$ be a 
non-elementary probability
distribution on $G$, and let $\tnu$ be a $\mu$-stationary Borel
measure on $\Xb$.

Then for almost every sample path $\omega = \seq{w_n}$, the sequence
of measures $\seq{w_n \tnu}$ converges to a measure $\tnu_\omega \in
\mathcal{P}(\Xb)$.
\end{lemma}

\begin{proof}
By Lemma \ref{lemma:non-atomic}, there is a $\mu$-stationary
probability measure $\nu$ on $\xr$ such that $\tnu$ is the pushforward of
$\nu$.  Applying Proposition \ref{P:martingale} to the action of $G$
on $\Xh$, the sequence $\seq{w_n \nu}$ converges to a measure
$\nu_\omega \in \mathcal{P}(\Xh)$, almost surely.  Moreover, since $\nu$ vanishes 
on $\xq$ and $\xq$ is $G$-invariant, the measures $w_n \nu$ also vanish on $\xq$ for each $w_n$; 
furthermore, by equation \eqref{eq:integral} the limit $\nu_\omega$ also vanishes on $\xq$ for almost every 
$\omega$. Now, note that $\xr$ is a countable intersection of open subsets
of $\Xh$, hence it is a Borel subset of $\Xh$, so the
weak-$*$ topology on $\mathcal{P}(\xr)$ (arising from $C_b(\xr)$)
is the relativization of the weak-$*$ topology on $\mathcal{P}(\Xh)$,  see e.g. 
\cite{ab}*{Theorem
  15.4}. Since $w_n \nu \to \nu_\omega$ a.s. in $\mathcal{P}(\Xh)$ and 
 both $w_n \nu$ and $\nu_\omega$ belong to $\mathcal{P}(\xr)$, this implies 
 that a.s. $w_n \nu \to \nu_\omega$ in the weak-$*$ topology of 
 $\mathcal{P}(\xr)$.
Finally, since $\phi$ is continuous, the
pushforward map $\mathcal{P}(\xr) \to \mathcal{P}(\Xb)$ is continuous,
hence $w_n \tnu  = \phi_* w_n \nu \to \phi_* \nu_\omega = \tnu_\omega$ as claimed.
\end{proof}

We wish to show that for $\P$-almost all sequences $\seq{w_n}$, the
measure $\tnu_\w$ is a $\delta$-measure, and as the limit exists, it
suffices to show this for any subsequence $(w_{n_k})_{k \in \N}$.  We
have already shown, by Proposition \ref{P:convsub}, that almost every
sequence $\seq{w_n}$ in $(\Omega, \P)$ has a subsequence $(w_{n_k})_{k
  \in \N}$ such that $(w_{n_k} x_0)_{k \in \N}$ converges to a point
in $\Xb$, so it suffices to show that if a sequence $\seq{g_n}$ has
the property that $\seq{g_n x_0}$ converges to $\lambda \in \Xb$, then
the measures $\seq{g_n \tnu}$ converge to $\delta_\lambda$.

\begin{lemma} \label{lemma:delta} %
Let $G$ be a non-elementary countable group of isometries of a
separable Gromov hyperbolic space $X$, 
and let $\seq{g_n}$ be a sequence in $G$ such that $g_n x_0 \to \lambda
\in \Xb$. Then for any non-atomic probability measure $\tnu$ on $\Xb$
there is a subsequence $(g_{n_k})_{k \in \N}$ such that the
translations $(g_{n_k} \tnu )_{k \in \N}$ converge in the weak-$*$
topology to a delta-measure $\delta_\lambda$, supported on $\lambda$.
\end{lemma}

\begin{proof}
By Lemma \ref{lemma:action}, there exists a subsequence $(n_k)_{k \in
  \N}$ and a point $x \in \Xb$ such that for all $ b \in \Xb, b \neq
x$, one has $g_{n_k} b \to \lambda$.  Since the measure $\tnu$ is
non-atomic, $\tnu(\{x\}) = 0$ and the claim follows by the dominated
convergence theorem.
\end{proof}

\begin{proposition} \label{prop:delta} %
Let $G$ be a countable group of isometries of a
separable Gromov hyperbolic space $X$, let $\mu$ be a 
non-elementary probability
distribution on $G$, and let $\tnu$ be a $\mu$-stationary
probability measure on $\Xb$.

Then for almost every sample path $\omega = \seq{w_n}$, there is a boundary
point $\lambda(\w) \in \Xb$ such that the sequence of measures
$\seq{w_n \tnu}$ converges in $\mathcal{P}(\Xb)$ to a delta-measure
$\delta_{\lambda(\w)}$.
\end{proposition}

\begin{proof}
By Proposition \ref{P:convsub}, almost every sample path $\w$ has a
subsequence which converges to some point $\lambda(\w) \in \Xb$.
Thus, by Lemma \ref{lemma:delta}, there exists a subsequence
$(w_{n_k})_{k \in \N}$ such that $(w_{n_k} \tnu)_{k \in \N} \to
\delta_{\lambda(\w)}$.  By Lemma \ref{lemma:mart2}, the
sequence $\seq{ w_{n} \tnu }$ has a limit, hence the limit must
coincide with $\delta_{\lambda(\w)}$.
\end{proof}

\subsection{Convergence to the boundary: end of proof} \label{S:convened}

We have shown that $\P$-almost every sequence $\seq{ w_n \tnu}$
converges to a $\delta$-measure on $\Xb$. Finally, we now show that
this implies that $\P$-almost every sequence $\seq{w_n x_0}$ converges
to some point in the Gromov boundary $\Xb$.

We start by showing that there are two shadows with disjoint closures and positive measure.

\begin{lemma} \label{L:disjshadows}
Let $X$ be a separable, Gromov hyperbolic space, and $\tnu$ a non-atomic probability measure on $\partial X$. Then 
there exist two shadows $S_1$, $S_2$ in $X$ such that their closures $U_i := \overline{S_i}^\delta$ in $\partial X$ are disjoint, 
and both have positive $\tnu$-measure. 
\end{lemma}

\begin{proof}
As $\partial X$ is a separable metric space, the support of $\tnu$ is
a non-empty closed set, and, since $\tnu$ is non-atomic, it contains
at least two points $\lambda_1$ and $\lambda_2$.  Now, for each
$\lambda_i$ we can choose a shadow $S_i$ in $X$ such that $\lambda_i$
is contained in the interior of the closure $U_i =
\overline{S_i}^\delta$, hence $\tnu(U_i) > 0$ for each $i = 1, 2$, and
such that the distance parameter of each $S_i$ is much larger than
$\gp{x_0}{\lambda_1}{\lambda_2}$, so that $U_1$ and $U_2$ are
disjoint.
\end{proof}

The next proposition completes the proof of Theorem \ref{T:conv}.

\begin{proposition} \label{P:topconv} %
Let $G$ be a countable group of isometries of a
separable Gromov hyperbolic space $X$, let $\mu$ be a non-elementary 
probability distribution on $G$, and let $\tnu$ be a $\mu$-stationary measure
on $\Xb$.

Suppose that the sequence $\seq{ w_n \tnu }$ converges to a
delta-measure $\delta_\lambda$ on $\Xb$, for some $\lambda \in
\Xb$. Then $\seq{ w_n x_0 }$ converges to $\lambda$ in $X \cup \Xb$.
\end{proposition}

\begin{proof}
Let $U_1 = \overline{S_1}^\delta$ and $U_2 = \overline{S_2}^\delta$ as in Lemma \ref{L:disjshadows}, applied to the stationary measure $\widetilde \nu$.
Set $\e = \min \{ \widetilde \nu(U_1), \widetilde \nu(U_2) \}$; as $U_1$ and $U_2$ both have
positive $\widetilde \nu$-measure, $\e$ is strictly greater than zero.  As the
sequence of measures $\seq{w_n \widetilde \nu}$ converges to the
delta-measure $\delta_\lambda$, for any shadow set $V = \overline{
  S_{x_0}(x, R) }^\delta$ containing $\lambda$ in its interior, there
is an $N$ such that 
$$w_n \tnu(V) \ge 1 - \e/2 \qquad \textup{for all }n
\ge N.$$
 As $w_n \widetilde{ \nu}(w_n U_1) = \tnu(U_1) \ge
\e$, the set $w_n U_1$ intersects $V$, 
hence $w_n S_1$ intersects the (slightly larger) shadow $S_0 = S_{x_0}(x, R + O(\delta))$, and similarly 
$w_n S_2$ intersects $S_0$. For each $i = 1, 2$, let us pick 
 $y_i \in S_i \cap w_n^{-1}S_0$ and denote $x_i = w_n y_i \in S_0$.

By disjointness of $U_1$ and $U_2$, there exists a constant $C$ such that 
$$\gp{x_0}{y_1}{y_2} \leqslant C$$
for each $y_1 \in U_1$ and $y_2 \in U_2$. Moreover, since $G$ acts by isometries, 
we get
$$\gp{w_n x_0}{x_1}{x_2} = \gp{x_0}{y_1}{y_2} \leqslant C$$
hence we can bound the distance from $w_n x_0$ to the geodesic $[x_1, x_2]$ as 
$$d_X(w_n x_0, [x_1, x_2]) =  \gp{w_n x_0}{x_1}{x_2} + O(\delta) \leqslant C + O(\delta);$$
note that the constant on the right-hand side depends only on $U_1$, $U_2$ and $\delta$, and not on $n$.

As $x_1$ and $x_2$ both lie in $S_0$,
by weak convexity (Corollary \ref{cor:weak convexity}), a geodesic
$[x_1, x_2]$ connecting them lies in a shadow $S_{x_0}(x, R +
O(\delta))$, and as $w_n x_0$ is a bounded distance from $[x_1, x_2]$, this implies that 
$w_n x_0$ lies in the slightly larger shadow $S_0^+ = S_{x_0}(x, R+ C + O(\delta))$, for all $n \ge N$.  

As this holds for all shadow sets
$V$ containing $\lambda$ in their interiors, this implies that $\seq{ w_n
x_0}$ converges to $\lambda$, as required.
\end{proof}

Theorem \ref{T:conv} implies the convergence statement in Theorem
\ref{theorem:convergence}, and so it remains to show
that the hitting measure $\tnu$ is the unique $\mu$-stationary measure
on $\Xb$, and the convolution measures $(\tmun{n})$ converge weakly to
$\tnu$.

By Proposition \ref{prop:delta}, for any $\mu$-stationary measure
$\tnu$ on $\Xb$, for $\P$-almost every sample path $\w$, the sequence
of measures $(w_n \tnu)_{n \in \N}$ converges to
$\delta_{\lambda(\w)}$, where $\lambda(\omega)$ is the limit point of
the sample path $(w_n x_0)_{n \in \N}$, hence it only depends on $\w$,
not on $\tnu$.  Thus, uniqueness follows from the integral formula
\eqref{eq:integral}.  Furthermore, 
\[ \tmun{n} = \int_\Omega w_n \delta_{x_0} \ d \P(\omega).  \]
We may take the limit as $n$ tends to infinity, and by the integral
formula \eqref{eq:integral}, the distribution of the limit points is
given by $\nu$, so $\tmun{n}$ weakly converges to $\tnu$.

\section{Applications} \label{section:applications}

In this section we use convergence to the boundary to show the
results on positive drift, sublinear tracking and translation length.

We will no longer use measures on the horofunction boundary, and so we
shall from now on simply denote by $\nu$ the hitting measure on the
Gromov boundary $\Xb$.  Moreover, given $S \subseteq X$, the symbol
$\overline{S}$ from now on will always mean the closure of $S$ in the
space $X \cup \partial X$.

\subsection{Hitting measures of shadows} \label{section:shadow measure}

We start by showing that the measure of a shadow tends to zero as
the distance parameter of the shadow tends to infinity.
In order to simplify notation, we shall denote 
$$Sh(x_0, r) := \{ S_{x_0}(gx_0, R) \ : \ g \in G, \ d_X(x_0, gx_0) - R \ge r \}$$
the set of shadows based at $x_0$, with centers on the orbit $Gx_0$
and with distance parameter $\ge r$.

\begin{proposition} \label{prop:shadow bound} %
Let $G$ be a countable group of isometries of a separable Gromov
hyperbolic space $X$.  Let $\mu$ be a non-elementary probability distribution on $G$,
and let $\nu$ be the hitting measure on $\Xb$. 
Then we have 
\[ \lim_{r\to \infty} \sup_{S \in Sh(x_0, r)} \nu( \overline{S }) =
0.  \]
\end{proposition}

This result also holds for the reflected measures $\check \mu_n$ and
$\check \nu$ determined by $\check \mu(g) = \mu(g^{-1})$, as if $\mu$
satisfies the hypotheses of Proposition \ref{prop:shadow bound}, then
so does $\check \mu$. 

\begin{proof}
The result follows from Propositions \ref{prop:shadow open} and
\ref{prop:bhm}, which says that a shadow centered at $g x_0$ of
distance parameter $r$ is contained in a ball of radius $Ce^{-\e r}$
in the metric $d_\e$ on $\Xb$, where $C$ is independent of $r$, and as
$\nu$ is non-atomic, the measure of a such a ball tends to zero as the
radius of the ball tends to zero.
\end{proof}

If $\mu$ has bounded range in $X$ then the argument from
\cite{Maher_exp}*{Lemma 2.10} shows that the $\nu$- and
$\mun{n}$-measures of a shadow $\overline{ S_{x_0}(g x_0, R) }$
decay exponentially in the distance parameter, i.e.  there are
positive constants $K$ and $c < 1$ such that
\begin{equation} \label{eq:shadow decay} 
\nu( \overline{S_{x_0}(g x_0, R)} ) \le K c^{d_X(x_0, g x_0) - R}.
\end{equation}

For $U$ a subset of $X$, let $H^+_x(U)$ denote  the probability that
a random walk starting at $x$ ever hits $U$ in forward time, i.e.
$$H^+_x(U) := \mathbb{P}( \exists \ n \ge 0 \ : \ w_n x \in U ).$$
Similarly,
let $H^-_x(U)$ be the probability that a random walk starting at $x$
ever hits $U$ in reverse time, i.e. the probability that $w_n x$ lies
in $U$ for some $n \le 0$.

\begin{proposition} \label{prop:hitting} %
Let $G$ be a countable group which acts by isometries on a separable
Gromov hyperbolic space $X$, and 
$\mu$ a non-elementary 
probability distribution
on $G$. 
Then
\[ \sup_{S \in Sh(x_0, r)} H^+_{x_0}(S) \to 0 \qquad \text{ as } r \to
\infty. \]
\end{proposition}

This immediately implies the same result with $H^+_{x_0}$ replaced by
$H^-_{x_0}$, by replacing $\mu$ with the reflected measure $\check
\mu(g) = \mu(g^{-1})$.

\begin{proof}
Suppose a sample path starting at $x_0$ hits a shadow $S_1 =
S_{x_0}(x, R)$ with $x = hx_0$ in forward time, at $g x_0$ say. Let $\gamma$ be a
geodesic from $x_0$ to $x$, and let $p$ be a nearest point on $\gamma$
to $g x_0$, as illustrated below in Figure \ref{pic:hitting shadows2}.

\begin{figure}[H]
\begin{center}
\begin{tikzpicture}

\tikzstyle{point}=[circle, draw, fill=black, inner sep=0pt,
minimum width=2.5pt]

\draw [thick] (-1, 0) node [point, label=below:$x_0$] {} -- (7, 0) node
[point, label=below:$x$] {};

\draw [thick] (6, 1.5) node [point, label=right:$g x_0$] {} -- (6,0) node
[point, label=below:$p$] {};

\draw [thick] (0, 1.5) node [point, label=right:$y$] {} -- (0,0) node
[point, label=below:$q$] {};

\draw [thick] (5, 2) -- (5, -2) node [right] {$S_1 = S_{x_0}(x, R)$};

\draw [thick] (1, 2) -- (1, -2) node [right] {$S_2 = S_{x_0}(x, R + A)$};

\end{tikzpicture}
\end{center} 
\caption{Nested shadows}
\label{pic:hitting shadows2}
\end{figure}
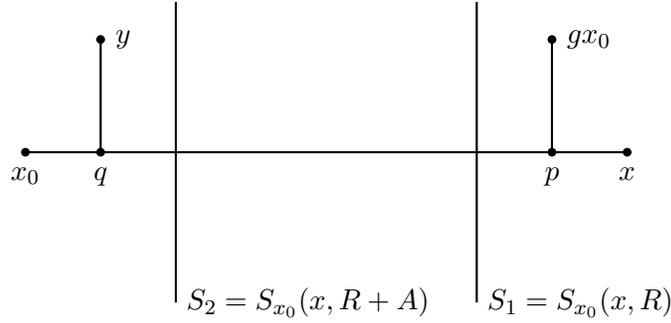

Consider the shadow $S_2 = S_{x_0}(x, R + A)$, for some fixed  $A > 0$ which we will choose later. 
The main idea is that if the random walk ever hits $S_1$, it will likely 
converge inside the closure of $S_2$, and the probability of that happening 
is small if the distance parameter of $S_2$ is large.

To make the idea precise, let $y$ be a point in the complement of $S_2$, and let $q$ be the nearest point projection
of $y$ to $\gamma$.  By Proposition \ref{prop:npp shadow}, $p$ is
within distance $O(\delta)$ of $S_1$, and $q$ is within distance
$O(\delta)$ of the complement of $S_2$, so the distance between $p$
and $q$ is at least $A + O(\delta)$. Now using Proposition
\ref{prop:npp2}, if $A \ge O(\delta)$, then the Gromov product satisfies
\[ \gp{g x_0}{x_0}{y} = d_X(g x_0, p) + d_X(p, q) + O(\delta) \ge A +
O(\delta). \]
Therefore, the complement of $S_2$ is contained in a shadow $S_3 =
S_{g x_0}(x_0, R')$ (where $R' = d_X(g x_0, x_0) - A + O(\delta)$) of distance parameter
$A + O(\delta)$.

Fix some positive $\e < 1$; then, since the measure of shadows tends to
zero as the distance parameter tends to infinity (Proposition
\ref{prop:shadow bound}), there is a number $A_0$ sufficiently large
such that $\nu(\overline{S}) \le \e$ for all shadows $S \in Sh(x_0, A_0)$ 
with distance parameter larger than $A_0$. 
As a consequence, if we choose $A$ such that $A + O(\delta) \ge A_0$ in the above construction, 
we have 
$$\nu(g^{-1}\overline{S_3}) = \nu(\overline{S_{x_0}(g^{-1}x_0, R')}) \leqslant \e$$
hence, since we proved $\partial X \setminus \overline{S_2} \subseteq \overline{X \setminus S_2} \subseteq \overline{S_3}$,  
\begin{equation}
\label{eq:largeS2}
\nu(g^{-1}\overline{S_2}) \geqslant 1 -\e.
\end{equation}
Now, by the Markov property of the random walk, 
the conditional probability of ending up in $\overline{S_2}$ 
after hitting an element $g x_0 \in S_1$ at time $k$ 
is given by 
$$\mathbb{P}(\lim_{n \to \infty} w_n x_0\in \overline{S_2} \ : \ w_k = g) = 
\mathbb{P}(\lim_{n \to \infty} w_n x_0\in g^{-1}\overline{S_2}) 
= \nu(g^{-1} \overline{S_2})$$
for each $k$ and $g$, and such probability is large by equation \eqref{eq:largeS2}. This implies the following 
lower bound on the probability of ending up in $\overline{S_2}$:
\begin{align*}
\mathbb{P}( \lim_{n \to \infty} w_n x_0 \in \overline{S_2}) 
& \geqslant \mathbb{P}(
\lim_{n \to \infty} w_n x_0 \in \overline{S_2} 
\text{ and } \exists \ n \ : \ 
w_n x_0 \in S_1) \\
\mathbb{P}( \lim_{n \to \infty} w_n x_0 \in \overline{S_2}) 
& \ge  \mathbb{P}(\exists \ n \ : \ w_n x_0 \in
S_1)(1-\e)
\end{align*}
which, by recalling the definitions of $\nu$ and $H^+_{x_0}$, becomes
\begin{equation} \label{eq:n-indep-estimate} 
H^+_{x_0}(S_1) \le \frac{1}{1 - \e} \nu( \overline{S_2} ). 
\end{equation}
Now, as the distance parameter of $S_1$ tends to $\infty$, 
so does the distance parameter of $S_2$, hence $\nu(\overline{S_2}) \to 0$
by Proposition \ref{prop:shadow bound},
and by the above equation $H^+_{x_0}(S_1)$ tends to $0$, as required.
\end{proof}

As $H^+_{x_0}(S_1)$ is an upper bound for $\mu_n(S_1)$ for any $n$, equation
\eqref{eq:n-indep-estimate} implies the following corollary.

\begin{corollary} \label{cor:mun estimate} %
Let $G$ be a countable group which acts by isometries on a separable
Gromov hyperbolic space $X$, and let $\mu$ be a non-elementary probability distribution
on $G$. 
Then there is a function $f(r)$, with $f(r) \to 0 $ as $r \to \infty$ such that
for all $n$ one has
\begin{equation} \label{e:mu_n estimate} 
\sup_{S \in Sh(x_0, r)} \mun{n}(S) \le f(r). 
\end{equation}
\end{corollary}

As the reflected random walk also satisfies the hypotheses of
Corollary \ref{cor:mun estimate}, we obtain a similar result for
$\check \mu_n$, though possibly for a different function $f$.

\begin{proposition} \label{prop:positive} %
Let $G$ be a non-elementary, countable group acting by isometries on a
separable Gromov hyperbolic space $X$, and let $\mu$ be a
probability distribution on $G$, whose support generates $G$ as a semigroup. 
Then there is a number $R_0$
such that for any $g, h \in G$ the closure of the shadow $S_{h x_0}(g x_0 ,
R_0)$ has positive hitting measure for the random walk determined by
$\mu$.
\end{proposition}

\begin{proof}
Let us first assume $h = 1$.
By Proposition \ref{prop:shadow open}, there is a constant $R_0$ such
that every shadow $S_1 = S_{x_0}(g x_0 , R_0)$ contains a limit point
$\lambda$ of $G x_0$ in the interior of its closure.
We may now follow the same argument as in Proposition \ref{prop:hitting}.
  Choose a shadow $S_2 = S_{x_0}(\widetilde{g} x_0, R_0)$
containing $\lambda$ such that $d_X(S_2, X \setminus S_1 )$ is at
least $A + O(\delta)$  (note that here
the roles of $S_1$ and $S_2$ are reversed with respect to Proposition \ref{prop:hitting}, as $S_2 \subseteq S_1$). In particular, this implies that for any point
$y \in S_2$, the complement of $S_1$ is contained in a shadow
$S_y(x_0, R)$ with distance parameter at least $A$.  By Proposition
\ref{prop:shadow bound} the measure of shadows tends to zero as the
distance parameter tends to infinity, so given a positive number $\e < 1$
we may choose $A$ sufficiently large so that $\nu(\overline{S}) \le
\e$ for all shadows $S \in Sh(x_0, A)$.

Let $\seq{g_n}$ be a sequence in $G$ such that $\seq{g_n x_0}$
converges to $\lambda$, and let $g_n$ be an element of the sequence
with $g_n x_0 \in S_2$, and let $k$ be such that $\mu_k(g_n) >
0$. Now using the Markov property of the random walk, the conditional
probability of converging to the closure of $S_1$, having hit $S_2$,
is at least $1 - \e$, and so
\[ \nu(\overline{S_1}) \ge (1 - \e) \mu_k(g_n), \]
which is positive, as required. Now, the case $h \neq 1$ can be reduced to the previous one; 
indeed, given $h \in G$, there exists by hypothesis an $n$ such that $\mu_n(h) > 0$, which implies
$$\nu(\overline{S_{hx_0}(gx_0, R_0)}) \ge \mathbb{P}(w_n = h) \  \mathbb{P}(\lim_m w_m x_0 \in \overline{S_{hx_0}(gx_0, R_0)} \ \mid \ w_n = h )$$
which by the Markov property of the walk equals 
$$\mu_n(h) \  \mathbb{P}(\lim_m w_m x_0 \in \overline{S_{x_0}(h^{-1}gx_0, R_0)})$$
and this is positive by the previous case.
\end{proof}

\subsection{Positive drift}

In this section we prove Theorem \ref{theorem:linear progress},
i.e. that the sample paths $\seq{w_n x_0}$ of the random walk have
positive drift in $X$.

It will be convenient to consider the $k$-step random walk $\seq{
  w_{kn} }$, and introduce the notation $x_i := w_{ki}x_0$ 
 for each $i$. Let $\chi^k_i \colon \Omega \to \R$ be the random
variable given by the distance in $X$ traveled by the sample path
from time $k(i-1)$ to time $ki$, i.e. 
\[ \chi^k_i(\omega) := d_X(w_{k(i-1)} x_0, w_{ki}
x_0) = d_X(x_{i-1}, x_i)
. \]
For fixed $k$, the $\chi^k_i$ are independent identically distributed
random variables with common distribution $\chi^k_1$.

Given a number $R$, we say a subsegment $[x_i, x_{i+1}]$
of the sample path is \emph{persistent} 
if the following three conditions are satisfied:
\begin{align}
& d_X(x_i, x_{i+1}) \ge 2R + 2C + C_0 \label{eq:c1} \\
& x_n \in S_{x_{i+1}}(x_{i}, R) \text{ for all } n
\le i \label{eq:c2} \\
& x_n  \in S_{x_{i}}(x_{i+1}, R) \text{ for all } n \ge
i+1 \label{eq:c3}
\end{align}

The constant $C$ in \eqref{eq:c1} is the weak convexity constant from
Corollary \ref{cor:weak convexity}, while $C_0$ will be depend on $\delta$ and will be chosen later. 
A persistent subsegment is illustrated in Figure \ref{pic:persistent} below.

\begin{figure}[H]
\begin{center}
\begin{tikzpicture}

\tikzstyle{point}=[circle, draw, fill=black, inner sep=0pt,
minimum width=2.5pt]

\draw [thick] (0, 0) node [point, label=above:$x_{i}$] {} -- (8,
0) node [point, label=below:$x_{i+1}$] {};

\draw [thick] (6, 2) -- (6, -2) node [right] {$S_{x_{i}}(x_{i+1}, R)$};

\draw [thick] (2, 2) -- (2, -2) node [left] {$S_{x_{i+1}}(x_{i}, R)$};

\draw (-1, -1) -- (0,0) -- (1.5,-1) -- (0.5, -1) -- (4.5, 2) -- (3,
1.5) -- (8, 1) -- (7, 1.5) -- (8,0) -- (9, 1);

\end{tikzpicture}
\end{center} 
\caption{A persistent segment of the sample path.}
\label{pic:persistent}
\end{figure}
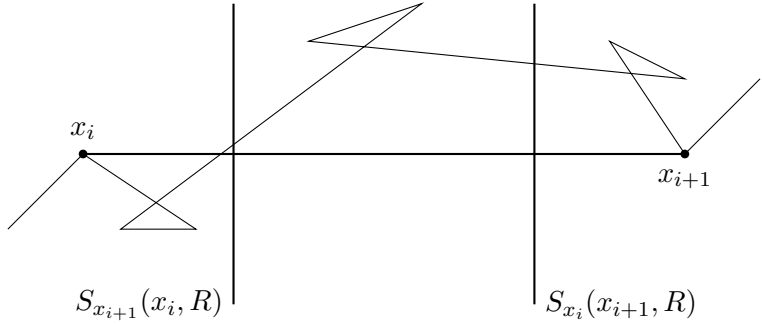

Choose an $\e$, with $0 < \e < \tfrac{1}{3}$. We now show that
given such a choice of $\e$, we may choose both $R$ and $k$
sufficiently large such that for any $i$ each of the three
conditions holds with probability at least $1 - \e$.

The probability that \eqref{eq:c2} holds is the same as the
probability that $w_{kn} x_0$ never hits the complement of the shadow 
$S_{x_{i+1}}(x_i, R)$ for any $n \le i$. 
As the complement of this shadow
is contained in a shadow
\[ S_i = S_{x_i}(x_{i+1}, R_i) \]
where 
$R_i = d_X(x_i, x_{i+1}) - R + O(\delta)$, 
the probability that \eqref{eq:c2} holds is at least
$$1 - \mathbb{P}( \exists \ n \le ki \ : \ w_n x_0 \in S_i )$$
which equals by the Markov property
\begin{equation} \label{eq:c2est}
1 - H^-_{x_0}(w_{ki}^{-1}S_i).
\end{equation}
The distance parameter of $w_{ki}^{-1}S_i$, which equals the distance parameter of $S_i$, is $R + O(\delta)$; 
hence, by Proposition \ref{prop:hitting}, 
we may
choose $R$ sufficiently large such that \eqref{eq:c2est} is at least
$1 - \e$.

A similar argument show that the probability that \eqref{eq:c3} holds
is at least
\begin{equation} \label{eq:c3est}
1- H^+_{x_0}(w_{k(i+1)}^{-1}S_{x_{i+1}}(x_i, R_i))
\end{equation}
and again we may choose $R$ sufficiently large such that
\eqref{eq:c3est} is at least $1 - \e$.

Finally, the probability that \eqref{eq:c1} holds is $\P(\chi^k_i \ge
2R + 2C + C_0 ) = \P(\chi^k_1 \ge 2 R + 2C + C_0 )$, 
since  $\chi_i^k$ and $\chi_1^k$ have the same law.
We have shown that almost every sample path
converges to a point in the Gromov boundary, so in particular, sample
paths are transient on bounded sets. This implies that for any $R$ and $\e$, there is a sufficiently large
 $k$, depending on
$R$ and $\e$, such that 
$$\P(\chi^k_1 \le 2 R + 2C + C_0 ) < \e$$
as required.
Therefore, for the choice of $\e, R$ and $k$ described above, the
probability that each condition holds individually is at least $1 -
\e$. The three conditions need not be independent, but the probability
that all three hold simultaneously is at least $\eta := 1 - 3 \e$, which is
positive as $\e < \tfrac{1}{3}$.

Thus, if we define for each $i$ the random variable $Y^k_i : \Omega \to \mathbb{R}$ 
\[ Y^k_i(\omega) := \left\{ 
\begin{array}{l} 
1 \text{ if } [x_i, x_{i+1}] \text{ is persistent} \\
0 \text{ otherwise.} \\
\end{array}
\right.  \]
we get that the $Y_i^k$ are identically distributed (but not
independent), with finite expectation since they are bounded; moreover, for each $i$
\begin{equation} \label{eq:positive}
\E(Y_i^k) \ge \eta > 0.
\end{equation}

We now show that the number of persistent segments lying between $x_0$
and $w_{kn} x_0$ 
gives a lower bound on the distance $d_X(x_0, w_{kn}
x_0)$.

Let $\gamma$ be a geodesic from $x_0$ to $x_n = w_{kn} x_0$, and suppose
that $[x_i, x_{i+1}]$ 
is a persistent subsegment of the
sample path. By \eqref{eq:c2}, $x_0$ lies in $S_{x_{i+1}}(x_{i}, R)$ for $i \ge 0$,  
and by \eqref{eq:c3}, $x_n$ lies in
$S_{x_{i}}(x_{i+1}, R)$ for $n \ge i+1$. As the two shadows are at least
distance $2C + C_0 - O(\delta)$ apart, the geodesic $\gamma$ has a
subsegment $\gamma_i$ of length at least $C_0 - O(\delta) \ge C_0/2$ which fellow
travels with $[x_{i}, x_{i+1}]$, and which is disjoint
from both $S_{x_{i+1}}(x_{i}, R+ C)$ and $S_{x_{i}}(x_{i+1}, R + C)$. 
Now let $[x_{j}, x_{j+1}]$ be a
different persistent subsegment. The same argument as above shows that
there is a subsegment $\gamma_j$ of $\gamma$ of length at least
$C_0/2$ which fellow travels with $[ x_{j}, x_{j+1} ]$. We now
show that $\gamma_i$ and $\gamma_j$ are disjoint subsegments of
$\gamma$. Up to relabeling, we may assume that $i < j$. Then both
$x_{j}$ and $x_{j+1}$ lie in $S_{x_{i}}(x_{i+1}, R)$,
and so by weak convexity, Corollary \ref{cor:weak convexity}, any
geodesic connecting them lies in $S_{x_{i}}(x_{i+1}, R +
C)$, and so in particular $\gamma_i$ and $\gamma_j$ are disjoint
subsegments of $\gamma$.  Therefore the distance $d_X(x_0, w_{kn}
x_0)$ is at least $C_0/2$ times the number of persistent
subsegments between $x_0$ and $w_{kn} x_0$.

We will now apply Kingman's subadditive ergodic theorem,
\cite{kingman}, using the following version from \cite{woess}*{Theorem 8.10}:

\begin{theorem}
\label{theorem:kingman}
Let $(\Omega,\P)$ be a probability space and $U:\Omega \to \Omega$ a
measure preserving transformation. If $W_n$ is a subadditive sequence
of non-negative real-valued random variables on $\Omega$, that is,
$W_{n+m} \le W_n + W_m \circ U^n$ for all $m,n \in \N$, and $W_1$ has
finite first moment, then there is a $U$-invariant random variable
$W_\infty$ such that \[ \lim_{n \to \infty} \tfrac{1}{n}W_n =
W_\infty \] $\P$-almost surely, and in $L^1(\Omega, \P)$.
\end{theorem}

In order to apply the theorem, let us define for each $n$ the variable
\[ 
Z^k_n := \sum_{i=0}^{n-1} Y_i^k = \#\{ 0 \le i \le n-1 \ : \  [x_i, x_{i+1}] \text{ is persistent} \}
\]
which gives the number of persistent subsegments along a given sample path from $x_0$ to $x_n = w_{kn} x_0$.
The random variables $(Z^k_n)_{n \in \N}$ are non-negative and have finite expectation, 
since $Z^k_n \le n$ for each $n$, and the sequence is subadditive by the Markov property.
Moreover, as expectation is additive, we get from equation \eqref{eq:positive}
\[ \E(Z_n^k) = \sum_{i=0}^{n-1} \E(Y^k_i) \ge n\eta \]
with $ \eta > 0$. We now apply Theorem \ref{theorem:kingman} taking as $\Omega$ the step space of the $k^{th}$-step 
random walk, $U$ the shift map, and the $Z_n^k$ as random variables (for fixed $k$); we get 
that the sequence $(\frac{1}{n} Z_n^k)_{n \in \N}$ converges almost surely and in $L^1$ to some 
random variable $Z_\infty^k$; moreover, since $U$ is ergodic, $Z_\infty^k$ must be constant almost everywhere, 
thus there exists a constant $A \ge 0$ such that 
$$\frac{1}{n} Z_n^k \to A$$
in $L^1$; finally, since $\E(Z_\infty^k) = \lim_n \E( \frac{1}{n} Z_n^k ) \ge \eta > 0$, we have that $A > 0$. 
Thus, since $Z_n^k$ is a lower bound for the distance $d_X(x_0, w_{kn} x_0)$, we get almost surely 
for the $k^{th}$-step random walk
$$\liminf_{n \to \infty} \frac{d_X(x_0, w_{kn} x_0)}{kn} \ge \frac{C_0}{2k} \liminf_{n \to \infty} \frac{1}{n} Z_n^k = \frac{A C_0}{2k} > 0$$
which proves the first part of
Theorem \ref{theorem:linear progress}, where we make no assumptions on
the moments of $\mu$.

For the second part of Theorem \ref{theorem:linear progress}, we
assume that $\mu$ has finite first moment with respect to the distance
function $d_X$. In this case, we can apply Kingman's Theorem directly
to $d_X(x_0, w_{kn} x_0)$, and we know that the limiting value $L$ is
positive, by the previous case.

Finally, if the support of $\mu$ is bounded in $X$, then the arguments
from \cite{Maher_exp} apply directly.

\subsection{Geodesic tracking}

We will now prove Theorem \ref{theorem:sublinear}, using the following
sublinearity result from Tiozzo \cite{tiozzo}.

\begin{lemma} \label{lemma:tiozzo}
Let $f \colon \Omega \to \R$ be a non-negative measurable function, $T : \Omega \to \Omega$ an ergodic, measure preserving transformation,
and suppose that 
\begin{equation} \label{eq:tiozzo}
g(\omega) = f(T \omega) - f(\omega) \text{ lies in } L^1(\Omega, \P).
\end{equation}
Then 
\[ \lim_{n \to \infty} \frac{f(T^n \omega)}{n} = 0  \]
for almost all $\omega \in \Omega$.
\end{lemma}

In order to apply it in this case, let us note that there are constants $Q$ and $c$, 
depending only on $\delta$, such that
any two distinct points in $\Xb$ are connected by a $(Q,
  c)$-quasigeodesic. We shall write $\Gamma(x, y)$ for the set of $(Q,
c)$-quasigeodesics connecting $x$ and $y$.
We then define $f : \Omega \to \mathbb{R}$ as 
\[ f(\omega) := \sup \{ d_X(x_0, \gamma) : \gamma \in
\Gamma(\omega_-, \omega_+) \}.  \]
As $\nu$ and $\check \nu$ are non-atomic, $\nu \cross \check \nu$ gives
measure zero to the diagonal in $\Xb \cross \Xb$, so
$\Gamma(\omega_-, \omega_+)$ is non-empty $(\nu \cross \check
\nu)$-almost surely, and the function $f(\omega)$ is well-defined $\P$-almost
surely.

Then by the triangle inequality, $\norm{ f(T \omega) - f(\omega) } \le
d_X(x_0, w_1 x_0)$, and so $g(\omega)$ lies in $L^1(\Omega, \P)$. 
Thus it follows from Lemma \ref{lemma:tiozzo} (where $T$ is the shift map 
on the step space) that
sample paths track
quasigeodesics sublinearly, i.e.
\[ \frac{d_X(w_n x_0, \gamma(\omega))}{n} \to 0, \text{ as } n \to
\infty, \text{ almost surely,}  \]
proving the first part of Theorem \ref{theorem:sublinear}.

To prove the second part, we need to show that, if $\mu$ has bounded support in $X$, then the tracking is in
fact logarithmic. We apply the argument from Blach\`ere, Ha\"issinsky and
Mathieu \cite{bhm}*{Section 3}, combined with our exponential decay of shadows.
We now give the details for the convenience of the reader.

We first show that the distribution of distances from the locations of
the sample path to the quasigeodesic satisfies an exponential decay
property.

\begin{proposition}
Let $G$ be a countable group of isometries of a separable Gromov
hyperbolic space $X$, and let $\mu$ be a non-elementary probability distribution on $G$.
There are positive constants $K$ and $c < 1$, which depend on $\mu$,
such that
\[ \P( d_X(w_n x_0, \gamma(\omega) ) \ge D ) \le K c^D,  \]
where $\gamma$ is a quasigeodesic ray from $x_0$ to the limit point in
$\Xb$ of $\seq{w_n x_0}$.
\end{proposition}

\begin{proof}
By the definition of $\gamma(\omega)$,
\[ d_X(w_n x_0, \gamma(\omega) )  = d_X(w_n x_0, [x_0, \lambda(\omega)] ),  \]
where $\lambda(\omega)$ is the limit point of $\seq{w_m x_0}$ in
$\Xb$, which exists for $\P$-almost every $\omega$. Applying the
isometry $w_n^{-1}$ gives
\[ d_X(w_n x_0, \gamma(\omega) ) = d_X( x_0, [w_n^{-1} x_0, w_n^{-1}
\lambda(\omega)] ). \]
Recall from \eqref{eq:gp estimate}, that the Gromov product
$\gp{x_0}{x}{y}$ may be estimated up to an error of $O(\delta)$ in
terms of the distance $d_X(x_0, [x, y])$, and a similar estimate holds
if one of $x$ or $y$ is a point in $\Xb$, and $[x, y]$ is a
quasigeodesic connecting them. This implies that
\[ d_X(w_n x_0, \gamma(\omega) ) = \gp{x_0}{w_n^{-1} x_0}{ w_n^{-1}
  \lambda(\omega)} + O(\delta). \]
So by the definition of a shadow, the condition
\[ d_X(w_n x_0, \gamma(\omega) ) \ge D \]
is equivalent to
\[ w_n^{-1}\lambda(\omega) \in \overline{ S_{x_0}(w_n^{-1}x_0, R) }, \]
where the parameter $R$ is given by 
\[ R = d_X(x_0, w_n^{-1} x_0) - D + O(\delta). \]

The boundary point $w_n^{-1} \lambda(\omega)$ only depends on the
increments of the random walk of index greater than $n$, so $w_n^{-1}
\lambda(\omega)$ and $w_n$ are independent. Furthermore, the
distribution of $w_n^{-1} \lambda(\omega)$ is equal to
$\nu$. Therefore
\[ \P( d_X(w_n x_0, \gamma(\omega) ) \ge D ) = \nu \left( \overline{
  S_{x_0}(w_n^{-1} x_0, R ) } \right), \]
and as $\mu$ has bounded range in $X$, we may use the exponential
decay estimate for shadows \eqref{eq:shadow decay}, which gives
\[ \P( d_X(w_n x_0, \gamma(\omega) ) \ge D ) \le K c^D, \]
as required.
\end{proof}

It follows immediately from the proposition above that there is a
constant $\kappa > 0$ such that
\[ \P( d_X( w_n x_0, \gamma(\omega) ) \ge \kappa \log n ) \le
\frac{1}{n^2}. \]
The logarithmic tracking result,
\[ \limsup \frac{ d_X(w_n x_0, \gamma(\omega) ) }{\log n} < \infty,
\text{ almost surely,} \]
then follows from the Borel-Cantelli lemma. This completes the proof of Theorem \ref{theorem:sublinear}.

\subsection{Translation length}

We briefly review some results about the translation length of
isometries, see e.g. Bridson and Haefliger \cite{bh} or
Fujiwara \cite{fujiwara}.

We start by observing that the translation length of an isometry $g$ may
be estimated in terms of the distance it moves the basepoint $x_0$, together
with the Gromov product of $g x_0$ and $g^{-1} x_0$.

\begin{proposition} \label{prop:hyperbolic} %
There exists a constant $C_0 >0$, which depends only on $\delta$, such that the following holds.
For any isometry $g$  of a $\delta$-hyperbolic space $X$, if $g$ satisfies the inequality
\begin{equation} \label{eq:hyperbolic}
d_X(x_0, g x_0) \ge 2 \gp{x_0}{g x_0}{g^{-1} x_0} + C_0, 
\end{equation}
then the translation length of $g$ is
\begin{equation} \label{eq:translation} 
\tau(g) = d_X(x_0, g x_0) - 2 \gp{x_0}{ g^{-1} x_0}{g x_0} +
O(\delta).  
\end{equation}
\end{proposition}

This is well known, but we provide a proof in the appendix for the
convenience of the reader.

In order to complete the proof of Theorem \ref{theorem:translation},
we shall now estimate the probability that the translation length is
small for a sample path of length $n$.  To apply the estimate
for translation length \eqref{eq:translation} we need a lower bound
for $d_X(x_0, w_n x_0)$, which is given by positive drift, and an
upper bound for the Gromov product $\gp{x_0}{w_n^{-1} x_0}{w_n x_0}$,
which we now obtain.

Let $m = \lceil n/2 \rceil$; we shall introduce the notation 
$u_m := w_m^{-1}w_n = g_{m+1}g_{m+2}\cdots g_n$, and we may think of $w_m x_0$ as an
approximate midpoint of the sample path from $x_0$ to $w_n x_0$, and of
$u_m^{-1} x_0 = w_n^{-1} w_m x_0$ as an approximate midpoint of the inverse sample
path from $x_0$ to $w_n^{-1} x_0$. 
 Note that for each $m$, the $G$-valued processes $w_m = g_1g_2\dots
g_m$ and $u_m := g_{m+1}g_{m+2}\cdots g_n$ are
independent.

Because of this independence, and the fact that the hitting measures are non-atomic, 
it is easy to prove the following upper bound on the Gromov product $\gp{x_0}{u_m^{-1} x_0}{ w_m x_0 }$.

\begin{lemma} \label{lemma:gp upper}
Let $G$ be a countable group of isometries of a separable Gromov
hyperbolic space $X$, and let $\mu$ be a non-elementary probability
distribution on $G$. If $l : \mathbb{N} \to \mathbb{N}$ is any function such that $l(n) \to \infty$ as $n \to \infty$, then we have 
%
\[ \P \left( \gp{x_0}{u_m^{-1} x_0}{w_m x_0} \le l(n)
\right) \to 1, \text{ as } n \to \infty, \]
for all $n$, where $m = \lceil n / 2 \rceil$.
\end{lemma}

\begin{proof}
By the definition of shadows,
\begin{align*} 
\P \left( \gp{x_0}{u_m^{-1} x_0}{w_m x_0} \le l(n)
\right) & = \P( u_m^{-1} x_0 \not \in S_{x_0}(w_m x_0, R) ),
\intertext{%
where $R = d_X(x_0, w_m x_0) - l(n)$.  As $w_m$ and
$u_m^{-1}$ are independent and the distribution of $u_m^{-1}$ is $\check \mu_{n-m}$,
}
\P \left( \gp{x_0}{u_m^{-1} x_0}{w_m x_0} \le l(n)
\right) & = 1 - \sum_{g \in G} \check \mu_{n - m}( S_{x_0}(g x_0, R) )
\mu_m(g) 
%
%
%
\intertext{%
Now, since the distance parameter of the shadows on the RHS is $l(n)$, using the estimate \eqref{e:mu_n estimate} gives
}
\P \left( \gp{x_0}{u_m^{-1} x_0}{w_m x_0} \le l(n)
\right) & \ge 1 - f(l(n))
\end{align*}
which tends to $1$ as $n \to \infty$.
%
%
\end{proof}

We will now use the fact that if the
Gromov products $\gp{x_0}{w_m x_0}{w_n x_0}$ and $\gp{x_0}{u_m^{-1} x_0}{w_n^{-1} x_0}$ are large, and the Gromov product
$\gp{x_0}{u_m^{-1} x_0}{w_m x_0}$ is small, then the two Gromov
products $\gp{x_0}{w_n^{-1} x_0}{w_n x_0}$ and $\gp{x_0}{w_n^{-1}
  x_0}{w_m x_0}$ are equal, up to bounded additive error depending
only on $\delta$.  This follows from the following lemma, which is a
standard exercise in coarse geometry. We omit the proof, but the
appropriate approximate tree is illustrated in Figure \ref{pic:gromov
  product}, with the points labeled according to our application.

\begin{lemma} \label{lemma:gp approx} %
For any four points $a, b, c$ and $d$ in a Gromov hyperbolic space
$X$, if there is a number $A$ such that $\gp{x_0}{a}{b} \ge A$,
$\gp{x_0}{c}{d} \ge A$ and $\gp{x_0}{a}{c} \le A - O(\delta)$ then
$\gp{x_0}{a}{c} = \gp{x_0}{b}{d} + O(\delta)$.
\end{lemma}

\begin{figure}[H] \begin{center}
\begin{tikzpicture}[scale=0.8]

\tikzstyle{point}=[circle, draw, fill=black, inner sep=0pt, minimum width=2.5pt]

\draw [thick] (0,0) node [point, label=below:$x_0$] {} -- (5, 0) node [point,
label=below:$w_m x_0$] {};

\draw [thick] (4,0) -- (4, 2) node [point, label=right:$w_n x_0$] {};

\draw [thick] (1,0) -- (1, -4) node [point, label=right:$u_m^{-1} x_0$]
{};

\draw [thick] (1,-3) -- (-1, -3) node [point, label=below:$w_n^{-1} x_0$] {};

\end{tikzpicture}
\end{center} 
\caption{Estimating the Gromov product.} \label{pic:gromov product}
\end{figure}
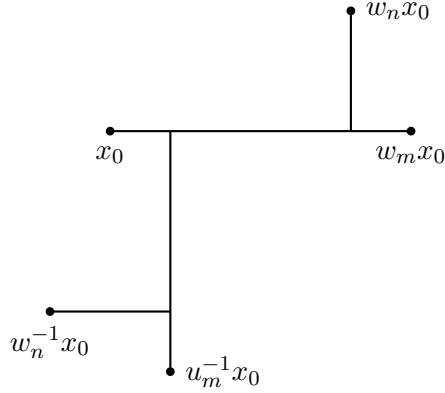

We now observe that with high probability, the Gromov
product $\gp{x_0}{w_m x_0}{w_n x_0}$ is large.
Recall that by linear progress there exists $L > 0$ such that 
\begin{equation}\label{eq:L}
\P\left( d_X(x_0, w_n x_0) \ge Ln \right) \to 1.
\end{equation}

\begin{lemma} \label{lemma:gp lower}
Let $G$ be a countable group of isometries of a separable Gromov
hyperbolic space $X$, and let $\mu$ be a non-elementary probability
distribution on $G$, and $L$ as in eq. \eqref{eq:L}. Then for any $l < L/2$ we have
\[ \P \left( \gp{x_0}{w_m x_0}{w_n x_0} \ge ln \right) \to
1, \text{ as } n \to \infty, \]
for all $n$, where $m = \lceil n / 2 \rceil$.
\end{lemma}

\begin{proof}
Note that by definition of shadows, we have the equality
\begin{align*}
\P \left( \gp{x_0}{w_m x_0}{w_n x_0} \ge l n \right) & = \P
\left( w_n x_0 \in S_{x_0}(w_m x_0, R ) \right),
\intertext{where $R = d_X(x_0, w_m x_0) - l n $.  Applying the isometry
$w_m^{-1}$ yields (recall $u_m = w_m^{-1}w_n$)} 
\P \left( \gp{x_0}{w_m x_0}{w_n x_0} \ge l n \right) & = \P
\left( u_m x_0 \in S_{w_m^{-1} x_0}(x_0, R ) \right).
\intertext{%
As the complement of a shadow is approximately a shadow (Corollary \ref{C:shadowcompl}),
}
\P \left( \gp{x_0}{w_m x_0}{w_n x_0} \ge l n \right) & \ge
\P \left( u_m x_0 \notin S_{x_0}( w_m^{-1} x_0, \widetilde{R}
) \right)
\intertext{%
with $\widetilde{R} = ln + O(\delta)$, and by conditioning
with $w_m = g$ we get
}
\P \left( \gp{x_0}{w_m x_0}{w_n x_0} \ge l n \right) & \ge
\sum_{g \in G} \P \left(u_m x_0 \notin S_{x_0}( w_m^{-1} x_0,
\widetilde{R} ) \ \middle\vert \ w_m = g \right) \P \left( w_m = g
\right).
\intertext{%
As $u_m$ and $w_m$ are independent, and the
  distribution of $u_m$ is $\mu_{n-m}$
  we have
}
\P \left( \gp{x_0}{w_m x_0}{w_n x_0} \ge ln \right) & \ge
\sum_{g \in G} \left( 1 - \mu_{n-m} \left( S_{x_0}( g^{-1} x_0,
\widetilde{R} ) \right) \right) \mu_m(g). \\
\intertext{%
Now, if we restrict to the set of $g$ such that $d_X(x_0, gx_0) \ge Ln/2$, then 
the distance parameter of the shadow is $d_X(x_0, g x_0) -
ln+ O(\delta) \ge \epsilon n  + O(\delta)$ for $\epsilon = \frac{L}{2}-l  > 0$, hence by the estimate
for $\mu_{n - m}$ in terms of the distance parameter, we get 
}
\P \left( \gp{x_0}{w_m x_0}{w_n x_0} \ge  l n \right) & \ge
\left(1 - f \left( \epsilon n + O(\delta) \right) \right)
\P\left(d(x_0, w_m x_0) \ge L n /2 \right).
\end{align*}
The result now follows by positive drift.
\end{proof}

The same argument applied to $w_n^{-1} x_0$, which has approximate midpoint
$u_m^{-1} x_0$, shows that
\[ \P \left( \gp{x_0}{u_m^{-1} x_0}{w_n^{-1} x_0} \ge l n \right) \to 1 \text{ as } n \to \infty \]
for any $l < L/2$.
Now using Lemma \ref{lemma:gp approx}, together with the lower bounds
on the Gromov products of $\gp{x_0}{w_m x_0}{ w_n x_0}$ and
$\gp{x_0}{u_m^{-1} x_0}{ w_n^{-1} x_0}$ from Lemma \ref{lemma:gp
  lower}, and the upper bound on the Gromov product $\gp{x_0}{u_m^{-1} x_0}{ w_m x_0}$ from Lemma \ref{lemma:gp upper} implies
\[ \P \left( \gp{x_0}{w_n^{-1} x_0}{w_n x_0} \le l(n)
\right) \to 1 \text{ as } n \to \infty. \]
for any function $l : \mathbb{N} \to \mathbb{N}$ such that $l(n) \to \infty$ as $n \to \infty$.
Applying this to the estimate for translation length
\eqref{eq:translation}, shows that the probability that
\[ \tau(w_n) \ge \tfrac{1}{2} Ln + O(\delta) \]
tends to $1$ as $n \to \infty$, as required.

If $\mu$ has bounded support in $X$, then this happens exponentially
fast, by \cite{Maher_exp}.

\section{The Poisson boundary for acylindrically hyperbolic groups}
\label{section:poisson}

In this section we prove Theorem \ref{theorem:poisson}, i.e. we show
that if the action of $G$ on $X$ is acylindrical and $\mu$ has finite
entropy and finite logarithmic moment, then in fact the Gromov
boundary with the hitting measure is the Poisson boundary. 

We shall assume from now on that $G$ is a non-elementary, countable group of isometries 
of a separable Gromov hyperbolic space $X$, and $\mu$ a probability measure 
on $G$ whose support generates $G$ as a semigroup.
Recall that the entropy of $\mu$ is $ H(\mu) := -\sum_{g \in G} \mu(g) \log \mu(g),$ 
and $\mu$ is said to have \emph{finite entropy} if $H(\mu) < \infty$. The
measure $\mu$ is said to have \emph{finite logarithmic moment} if 
$$\sum_{g \in G} \mu(g) |\log d_X(x_0, g x_0)| < \infty.$$ 

Let us recall the definition of acylindrical action, which is due to
Sela \cite{sela} for trees, and Bowditch \cite{bowditch} for general
metric spaces.

\begin{definition}
We say a group $G$ acts \emph{acylindrically} on a Gromov hyperbolic
space $X$, if for every $K \ge 0$ there are numbers 
$R = R(K)$ and $N = N(K)$ 
such that for any pair of points $x$ and $y$ in $X$,
with $d_X(x, y) \ge R$,
there are at most $N$ group
elements $g$ in $G$ such that $d_X(x, gx) \le K$ and $d_X(y, gy) \le
K$. 
\end{definition}

For a discussion and several examples of acylindrical actions on
hyperbolic spaces, see \cite{osin}.

The proof will use Kaimanovich's strip criterion from
\cite{kaimanovich}.  Briefly, the criterion uses the existence of
``strips'', that is subsets of $G$ which are associated to each pair
of boundary points in a $G$-equivariant way.

In order to apply the criterion, however, one also needs to control
the number of elements in the strips; in fact, we will show that for
each strip the number of elements whose images in $X$ lie in a ball of
radius $r$ can grow at most linearly in $r$.  In a proper space, one
may often choose the strips to consist of all geodesics connecting the
endpoints of the sample path, but in our case, this usually gives
infinitely many points in a ball of finite radius. Instead, we observe
that by recurrence, the sample path returns close to a geodesic
connecting its endpoints for a positive density of times $n \in
\Z$. Using this it can be shown that there are infinitely many pairs
of locations $w_n x_0$ and $w_{n+m} x_0$, where the sample path has
gone a definite distance along the geodesic in bounded time. In fact,
we may choose a suitable group element $v$, and look at all group
elements $g$ whose orbit points $g x_0$ are close to a geodesic
$\gamma$, such that both $g x_0$ and $g v x_0$ are close to
$\gamma$. We shall call the collection of such group elements
\emph{bounded geometry} elements, and we will choose our strips to
consist of these elements. We will use acylindricality to show that
this set is locally finite, and in fact the intersection of its image
in $X$ with $B_X(x_0, r)$ grows at most linearly with $r$.  Let us now
make this precise.

\subsection{Bounded geometry points}

Let $v \in G$ be a group element, $K, R$ two constants, and $\alpha,
\beta \in \partial X$ two boundary points.

We say that a group element $g$ has \emph{$(K,R,v)$-bounded geometry}
with respect to the pair of boundary points $\alpha$, $\beta
\in \partial X$ if the three following conditions hold:

\begin{enumerate}
\item $d_X(gx_0, gvx_0) \ge R$; 
\item $\alpha$ belongs to the interior of the closure (in $X \cup \Xb$) of $S_{gvx_0}(gx_0, K)$; 
\item $\beta$ belongs to the interior of the closure of $S_{gx_0}(gvx_0, K)$.
\end{enumerate}

This is illustrated in Figure \ref{pic:bounded geometry} below.  We
shall write $\bg(\alpha, \beta)$ for the set of bounded geometry
elements determined by $\alpha$ and $\beta$ (or $\bg_{K, R, v}(\a,
\b)$ if we want to explicitly keep track of the constants).  This
definition is $G$-equivariant, i.e. $g \bg(\a, \b) = \bg(g \a, g \b)$
for any $g \in G$. We will refer to the image of a bounded geometry
element in $X$ under the orbit map as a \emph{bounded geometry point}.

\begin{figure}[H] \begin{center}
\begin{tikzpicture}[scale=0.5]

\tikzstyle{point}=[circle, draw, fill=black, inner sep=0pt, minimum width=2.5pt]

\draw (0,0) circle (5cm);

\draw (60:5) node [point, label=right:$\beta$] {}; 
\draw (-60:5) node [point, label=right:$\alpha$] {};

\draw (85:5) .. controls (1, 0.5) .. (25:5) node [right] {$S_{g x_0}(g v
  x_0, K)$};

\draw (-140:5) .. controls (1.5, 0.75) .. (-40:5) node [right] {$S_{g v x_0}(g
  x_0, K)$};

\draw (1, -1) node [point, label=right:$g x_0$] {};

\draw (2, 2) node [point, label=right:$g v x_0$] {};

\draw (-2, 2) node {$X$};

\draw (-5, 0) node [left] {$\Xb$};

\end{tikzpicture}
\end{center} 
\caption{A bounded geometry point $g x_0$ in $\bg_{K, R, v}(\alpha,
  \beta)$.} \label{pic:bounded geometry}
\end{figure}
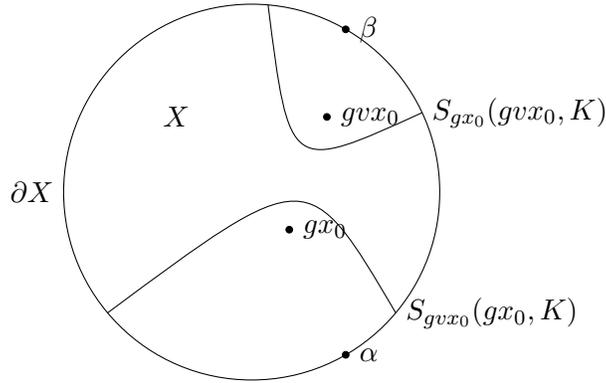

We say a set of group elements $\bg$ is \emph{locally finite} if the
set $\bg x_0 \cap B_X(x, r)$ is finite for all $x \in X$ and all $r
\ge 0$, and that $\bg$ has \emph{linear growth} if there is a constant
$C$ such that for all $r \ge 0$
\[ \norm{\bg x_0 \cap B_X(x_0, r)} \le Cr.  \]

We now show that the set of bounded geometry elements has linear
growth.

\begin{proposition} \label{p:linear} %
There exists $K_0$ such that for any $K \ge K_0$, there exists $R_0$,
such that for any $R \ge R_0$, there exists a constant $C$ such that we
have the estimate
\[ \norm{B_X(x_0, r) \cap \bg_{K, R, v}(\a, \b) x_0 } \le  C r. \]
for any $\alpha, \beta \in \partial X$, any $r > 0$  
and any group element $v \in G$.
\end{proposition}

In order to prove the proposition, let us start by proving that the
number of bounded geometry elements in a ball $B_X(x, 4K)$ is bounded
in terms of $K$.

\begin{lemma} \label{prop:local bound} %
There exists $K_0$ such that, for any $K \ge K_0$, there exists $R_0$ such that for any $R \ge R_0$ 
and any group element $v \in G$, we have the estimate
\[ \norm{ B_X(x, 4 K) \cap \bg_{K, R, v}(\a, \b) x_0} \le N(22K) \]
for any pair of boundary points $\a$ and $\b$, and any point $x \in
X$.
\end{lemma}

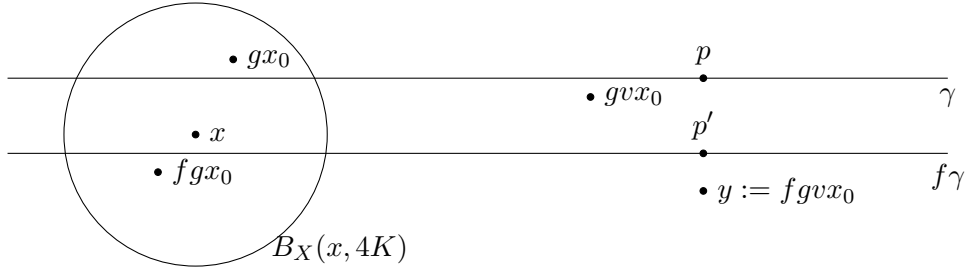
\begin{figure}[H] \begin{center}
\begin{tikzpicture}[scale=0.5]

\tikzstyle{point}=[circle, draw, fill=black, inner sep=0pt, minimum width=2.5pt]

\draw (-1, 3.5) -- (24, 3.5) node [below] {$\gamma$};
\draw (-1, 1.5) -- (24, 1.5) node [below] {$f \gamma$};

\filldraw[black] (4,2) node [point, label=right:$x$] {};

\draw (4, 2) circle (3.5cm);

\draw (4, 2) +(-60:3.5) node [right] {$B_X(x, 4 K)$};

\filldraw[black] (5,4) node [point, label=right:$g x_0$] {};
\filldraw[black] (3,1) node [point, label=right:$f g x_0$] {};

\filldraw[black] (14.5,3)   node [point, label=right:$g v x_0$] {};
\filldraw[black] (17.5,0.5) node [point, label={right:$y := f g v x_0$}] {};
\filldraw[black] (17.5,3.5) node [point, label=above:$p$] {};
\filldraw[black] (17.5,1.5) node [point, label=above:$p'$] {};

\end{tikzpicture}
\end{center} 
\caption{Bounded geometry points in $B_X(x, 4K)$.}\label{pic:geodesic}
\end{figure}

\begin{proof}
Recall from Section \ref{section:qg} that we have chosen $Q$ and $c$
to be two numbers such that every pair of points in the Gromov
boundary $\Xb$ are connected by a continuous $(Q,
c)$-quasigeodesic. We shall choose $L$ to be a Morse constant for the
$(Q, c)$-quasigeodesics, i.e. for any pair of points $x$ and $y$ in a
$(Q, c)$-quasigeodesic $\gamma$, the segment of $\gamma$ between $x$
and $y$ is contained in an $L$-neighbourhood of any geodesic
connecting $x$ and $y$.  

Let $g$ be a bounded geometry element, and write $S_1$ for $S_{g v
  x_0}( g x_0, K)$ and $S_2$ for $S_{g x_0}(g v x_0, K)$.  As $g$ has
bounded geometry, each $(Q,c)$-quasigeodesic from $\alpha$ to $\beta$
passes within distance $L + K + O(\delta)$ of both $gx_0$ and
$gvx_0$. Therefore we may choose $K_0$ to be sufficiently large such
that for any $K \ge K_0$ the distance from any bounded geometry point
$g x_0$ to a $(Q, c)$-quasigeodesic connecting $\alpha$ and $\beta$ is
at most $2 K$. Furthermore, we will choose $K_0$ to be larger than the
quasigeodesic constants $Q$ and $c$, and also larger than the Morse
constant $L$.

Let $g x_0$ and $g'x_0 \in B_X(x, 4 K) \cap \bg_{K, R, v}(\a, \b) x_0$
be two bounded geometry elments with respect to the same boundary
points $(\alpha, \beta)$, and the same element $v$.  We may write $g'
= f g$, for some group element $f \in G$. This is illustrated in
Figure \ref{pic:geodesic} above.  The isometry $f$ moves the point $g
x_0$ distance at most $8 K$.

Now let $\gamma$ be a $(Q,c)$-quasigeodesic joining $\alpha$ to
$\beta$.  By construction, both $\gamma$ and $f\gamma$ have endpoints
in $fS_1$ and $fS_2$, hence they both pass within distance $2K$ from
both $fgx_0$ and $fgvx_0$.  Let us now consider $y := fgvx_0$; we now
show that the isometry $f$ also moves the point $y$ a bounded
distance, which yields the claim by definition of an acylindrical
action.

Let $p$ be a closest point on $\gamma$ to $y$, 
and $p'$ be a closest point on $f \gamma$ to $y$, 
so by construction 
$$d_X(p, p') \le d_X(p, y) + d_X(p', y) \le 4K.$$
As $d_X( g x_0, f g x_0) \le 8 K$, this implies that 
$$\norm{d_X(g  x_0, p) - d_X(f g x_0, p')} \le 12 K.$$
Therefore, since $d_X(fgx_0, fp) = d_X(gx_0, p)$, and both $fp$ and
$p'$ lie on the quasigeodesic $f\gamma$, we have
$$d_X(fp, p') \le 16K + 2L \le 18K$$ 
and so 
$$d_X(y, fy) \le d_X(y, p') + d_X(p', fp) + d_X(fp, fy) \le 22 K.$$
Therefore, $f$ moves each of $x$ and $y$ distance at most $22K$, and so
by acylindricality there are at most $N(22 K)$ possible choices
for $f$, as long as $R \ge R(22 K)$,
as required.
\end{proof}

\begin{proof}[Proof of Proposition \ref{p:linear}] %
Let $\gamma$ be a $(Q,c)$-quasigeodesic connecting $\alpha$ and
$\beta$.  We shall choose the number $K_0$ to be the same as the
number $K_0$ from Lemma \ref{prop:local bound}. Then $K$ is
sufficiently large such that every element of $\bg(\alpha, \beta)$ has
an image in $X$ which lies within distance at most $2K$ of $\gamma$,
and any pair of points $\gamma(n)$ and $\gamma(n+1)$ on the
quasigeodesic are distance at most $2K$ apart. Therefore $\bg(\alpha,
\beta) x_0$ is covered by balls of the form $B_X(\gamma(n), 4K)$ for
$n \in \mathbb{N}$, and the claim follows from applying Lemma
\ref{prop:local bound} to each of these balls.
\end{proof}

\subsection{Recurrence and the strip criterion}

Given a bi-infinite sample path $\seqz{w_n}$, we shall define the
forward and backward limit points to be
\[ \lambda_+(\w) := \lim_{n \to \infty} w_n x_0, \text{ and }
\lambda_-(\w) := \lim_{n \to \infty} w_{-n} x_0.  \]
As the forward and backward random walks converge to the Gromov
boundary, these limit points are defined for $\P$-almost all $\w$, and
the joint distribution of the pair $(\lambda_+(\w), \lambda_-(\w))$ is
$\nu \cross \rnu$.

For any bi-infinite sequence $\w \in \Omega = G^\mathbb{Z}$, we define
$\bg(\w)$ to be the set $\bg(\lambda_+(\w), \lambda_-(\w))$ of bounded
geometry elements determined by the limit point $\lambda_+(\omega)$ of the
forward random walk and the limit point $\lambda_-(\omega)$ of the
backward random walk.

Finally, we show that we can choose $K$, $R$, and $v$ such that the
set of bounded geometry elements is non-empty and locally finite for
$\nu \cross \rnu$-almost all $(\a, \b) \in \Xb \cross \Xb$.

\begin{proposition} \label{prop:non-empty} %
There are constants $K, R$ and a group element $v \in G$ such that the
set
$$\bg_{K, R, v}(\a, \b)$$
of bounded geometry elements has linear growth and is non-empty (in
fact, infinite) for $\nu \cross \rnu$-almost all pairs $(\a, \b)$.
\end{proposition}

\begin{proof}
By Proposition \ref{prop:positive}, we can choose $K$ large enough so
that for any group element $v$, the closure of the shadow $S =
S_{x_0}(v x_0, K)$ has positive $\nu$-measure, and the closure of the
shadow $S' = S_{v x_0}(x_0, K)$ has positive $\rnu$-measure.  Thus,
the probability that the group identity element $1$ lies in $\bg(\w)$
is positive, because
\[ \P \left( \omega \ : \  1 \in \bg(\omega) \right) = \nu(\overline{S}) \rnu(\overline{S'}) = p > 0. \]
Consider the probability that the location of the random walk $w_n$
lies in $\bg(\omega)$, i.e.
\begin{align*}
&
\P \left( \omega \ : \  w_n \in \bg(\lambda_+(\omega), \lambda_-(\omega)) \right).  \\
\intertext{%
By $G$-equivariance, this is equal to
}
& \P \left( \omega  \ : \ 1 \in \bg(w_n^{-1} \lambda_+(\omega) , w_n^{-1} \lambda_-(\omega)) \right), \\
\intertext{%
and by definition of the shift map this is equal to
}
& \P  \left( \omega  \ : \  1 \in \bg(T^n \omega ) \right). \\
\intertext{%
As the shift map preserves the measure $\P$, this is equal
  to
}
& \P \left( \omega  \ : \ 1 \in \bg( \omega ) \right) = p > 0.
\end{align*}
Therefore the events $\{ \w \in \Omega \ : \ w_n \in \bg(\w)\}$ occur
with the same positive probability $p$, though they are not
independent. By ergodicity of the shift map, the proportion of
locations $\{w_1, \ldots w_N \}$ satisfying $w_n \in \bg(\w)$
converges to $p$ as $N$ tends to infinity. As $\seq{ w_n x_0 }$
converges to the boundary $\P$-almost surely, $\bg(\w)$ contains
infinitely many elements $\P$-almost surely.
\end{proof}

We remind the reader of Kaimanovich's strip criterion from
\cite{kaimanovich}*{Theorem 6.4}. We shall write $B_G(1, r)$ for all
group elements whose image in $X$ under the orbit map is distance at
most $r$ from the basepoint $x_0$, i.e.
\[ B_G(1, r) = \{ g \in G \mid d_X(x_0, g x_0) \le r \}.  \]

\begin{theorem}
Let $\mu$ be a probability measure with finite entropy on $G$, and let
$(\partial X , \nu )$ and $(\partial X , \rnu )$ be $\mu$- and
$\rmu$-boundaries, respectively. If there exists a measurable
$G$-equivariant map $S$ assigning to almost every pair of points $(\a
, \b ) \in \partial X \cross \partial X$ a non-empty ``strip'' $S(\a ,
\b ) \subset G$, such that for all $g$
\[ \frac{1}{n} \log \left| S(\a , \b ) g \cap B_G( 1, d_X(x_0, w_n
x_0) )\right| \to 0 \qquad \text{ as } n \to \infty, \]
for $(\nu \cross \rnu )$-almost every $(\a , \b ) \in \partial X
\cross \partial X$, then $(\partial X , \nu )$ and $(\partial X ,
\rnu )$ are the Poisson boundaries of the random walks $(G, \mu)$ and
$(G, \rmu)$, respectively.
\end{theorem}

In order to prove Theorem \ref{theorem:poisson}, we define the strip
$S(\alpha, \beta)$ as the set $\bg_{K, R, v}(\alpha, \beta)$ of
bounded geometry elements. By right multiplication by $g^{-1}$, the
set 
\[ S(\a , \b ) g \cap B_G( 1, d_X(x_0, w_n x_0) ) \]
has the same cardinality as 
\[ S(\a , \b ) \cap B_G( 1, d_X(x_0, w_n x_0) )g^{-1}. \]
Furthermore, 
\[ B_G( 1, d_X(x_0, w_n x_0) )g^{-1} \subset B_G( 1,
d_X(x_0, w_n x_0) + d_X(x_0, g x_0)), \]
and so
\[ \left| S(\a , \b ) g \cap B_G(1 , d_X(x_0, w_n x_0)) \right| \le
\left| S(\a , \b ) \cap B_G(1 , d_X(x_0, w_n x_0) + d_X(x_0, g x_0))
\right|. \]
Proposition \ref{prop:non-empty} shows that there are suitable choices
of $K, R$ and $v$ such that the sets of bounded geometry elements are
non-empty almost surely and have linear growth, so there is a number
$K$ such that
\[ \left| S(\a , \b ) g \cap B_G(1 , d_X(x_0, w_n x_0)) \right| \le K
( d_X(x_0, w_n x_0) +d_X(x_0, g x_0) ). \]
Therefore, it suffices to show that almost surely $\log d_X(x_0, w_n x_0)  / n \to
0$ as $n \to \infty$, and this follows from the fact that $\mu$ has
finite logarithmic moment, as we now briefly explain. Finite
logarithmic moment implies that $\log d_X(x_0, g_n x_0)  / n \to 0$
almost surely, and so for any $\e > 0$, we have $d_X(x_0, g_n x_0) \le
e^{\e n}$ for all $n$ sufficiently large. By the triangle inequality
\[ d_X(x_0, w_n x_0) \le d_X(x_0, g_1 x_0) + \cdots + d_X(x_0, g_n
x_0), \]
and so 
\[ \log  d_X(x_0, w_n x_0)  \le \log n + \e n. \] 
As this holds for all $\e > 0$, this implies that $\log d_X(x_0, w_n
x_0)  / n \to 0$ as $n \to \infty$, as required.

Finally, the statement that the map $S$ is measurable means that for
any $g \in G$, the set
$$\{ (\alpha, \beta) \in \partial X \times \partial X
\mid g \in S(\alpha, \beta ) \}$$
is a Borel set; this holds, since by definition, $g \in S(\alpha,
\beta)$ if and only if $(\alpha, \beta)$ belongs to the product of the
closures of two shadows, which is closed, hence Borel. This completes
the proof of Theorem \ref{theorem:poisson}.

\appendix
\section{Estimating translation length}

In this section we provide a proof of Proposition
\ref{prop:hyperbolic}, which estimates the translation length in terms
of the distance an isometry moves the basepoint, and the Gromov product.

\begin{prop} 
There exists a constant $C_0 >0$, which depends only on $\delta$, such that the following holds.
For any isometry $g$  of a $\delta$-hyperbolic space $X$, if $g$ satisfies the inequality
\begin{equation} 
d_X(x_0, g x_0) \ge 2 \gp{x_0}{g x_0}{g^{-1} x_0} + C_0, 
\end{equation}
then the translation length of $g$ is
\begin{equation} 
\tau(g) = d_X(x_0, g x_0) - 2 \gp{x_0}{ g^{-1} x_0}{g x_0} +
O(\delta).  
\end{equation}
\end{prop}

\begin{proof}
We start by showing that if $\gamma$ is a geodesic segment from $x_0$
to $g^n x_0$, then $g^k x_0$ is contained in a bounded neighbourhood of $\gamma$, for all $0 \le k \le n$. We
shall write $x_k$ for $g^k x_0$. Note that, since the action is isometric, we have for each $k$
\begin{equation} \label{eq:indepk}
d_X(x_{k+1}, x_{k+2}) = d_X(x_0, gx_0), \text{ and } 
\gp{x_{k+1}}{x_k}{x_{k+2}} = \gp{x_0}{gx_0}{g^{-1}x_0}. 
\end{equation}

\begin{claim}
Let $\gamma$ be a geodesic from $x_0$ to $x_n$. Then 
\begin{equation} \label{eq:bounded nbd} 
d_X(x_k, \gamma) \le \gp{x_0}{g x_0}{ g^{-1} x_0} + O(\delta), 
\end{equation}
for all $0 \le k \le n$.
\end{claim}

\begin{proof}[Proof (of claim).]
Let $p_k$ be a nearest point on $\gamma$ to $x_k$, and let $x_k$ be an
element of $\{x_k\}_{k=0}^n$ furthest from $\gamma$.  Consider the
quadrilateral formed by $x_{k-1}, x_{k+1}, p_{k+1}$ and $p_{k-1}$, as
illustrated below in Figure \ref{pic:bounded neighbourhood}.

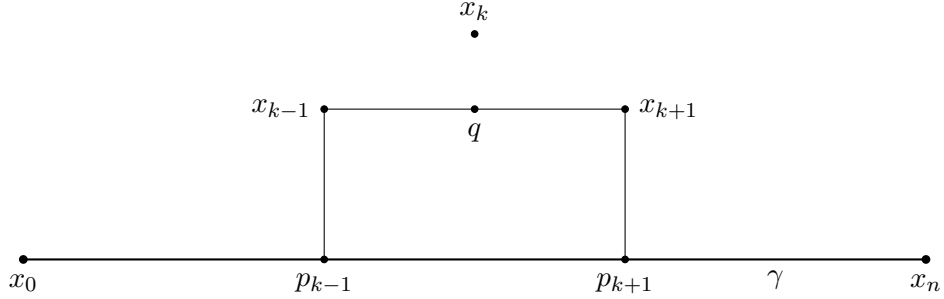
\begin{figure}[H]
\begin{center}
\begin{tikzpicture}

\tikzstyle{point}=[circle, draw, fill=black, inner sep=0pt,
minimum width=2.5pt]

\draw [thick] (-2, 0) node [point, label=below:$x_0$] {} 
              -- (6, 0) 
              -- (10, 0) node [midway, below] {$\gamma$}  node [point, label=below:$x_n$] {};

\draw (2, 0) node [point, label=below:$p_{k-1}$] {};
\draw (6, 0) node [point, label=below:$p_{k+1}$] {};

\draw (2, 0)
      -- (2, 2) node [point, label=left:$x_{k-1}$] {}
      -- (4, 2) node [point, label=below:$q$] {}
      -- (6, 2) node [point, label=right:$x_{k+1}$] {} 
      -- (6, 0);

\draw (4, 3) node [point, label=above:$x_k$] {};

\end{tikzpicture}
\end{center} 
\caption{The points $x_k = g^k x_0$ lie in a bounded neighbourhood of
  $\gamma$.}
\label{pic:bounded neighbourhood}
\end{figure}

Let $q$ be a nearest point to $x_k$ on a geodesic segment from $x_{k-1}$ to $x_{k+1}$.
By the estimate for the Gromov product in terms of distance to a
geodesic \eqref{eq:gp estimate},
\[ \gp{x_k}{ x_{k-1} }{x_{k+1} } = d_X(x_k, [x_{k-1}, x_{k+1}]) + O(\delta), \]
and as $g$ is an isometry, this implies that
\[ d_X(x_k, q) = \gp{x_0}{g x_0}{ g^{-1} x_0} + O(\delta). \]
By thin triangles, the point $q$ lies within distance $2 \delta$ of at
least one of the other three sides of the quadrilateral.  Suppose $q$
lies within $2 \delta$ of a geodesic from $x_{k-1}$ to $p_{k-1}$. 

Since $x_k$ is the furthest point from the geodesic $\gamma$, we have 
$d_X(x_k, \gamma) \ge d_X(x_{k-1}, \gamma)$, and since $q$ lies in a $2\delta$-neighbourhood of 
the geodesic from $x_{k-1}$ to $p_{k-1}$, we have 
\begin{equation}
\label{E:xk}
d_X(x_{k-1}, q) \le d_X(x_k, q) + O(\delta).
\end{equation}

On the other hand, if we now assume \eqref{eq:hyperbolic} and apply the reverse triangle inequality \eqref{eq:reverse triangle},
we get 
\begin{align*}
d_X(x_k, x_{k-1}) & = d_X(x_k, q) + d_X(x_{k-1}, q) + O(\delta) \\ 
d_X(x_k, x_{k-1}) & \ge 2d_X(x_k, q) + C_0 + O(\delta) 
\end{align*}
hence 
$$d_X(x_{k-1}, q) \ge d_X(x_k, q) + C_0 + O(\delta),$$ 
which contradicts \eqref{E:xk} if $C_0$ is large enough (depending only on $\delta$).
The same argument applies if $q$ lies within $2 \delta$ of a
geodesic from $x_{k+1}$ to $p_{k+1}$. Therefore, $q$ lies within $2
\delta$ of $\gamma$, and so $d_X(x_k, \gamma) \le \gp{x_0}{g x_0}{
  g^{-1} x_0} + O(\delta)$, as required.
\end{proof}

Consider a pair of adjacent points $x_k$ and $x_{k+1}$.  Combining
\eqref{eq:hyperbolic}, \eqref{eq:bounded nbd}, and the triangle inequality, gives
\begin{equation} \label{eq:translation2} 
d_X(p_k, p_{k+1}) \ge d_X(x_0, g x_0) - 2\gp{x_0}{g x_0}{g^{-1}
  x_0} + O(\delta).  
\end{equation}

Consider a triple of consecutive points, $x_k, x_{k+1}$ and
$x_{k+2}$. If their corresponding nearest point projections $p_k,
p_{k+1}$ and $p_{k+2}$ to $\gamma$ do not lie in the same order, then
using Proposition \ref{prop:npp2} repeatedly one gets, if $p_{k+2}$ lies in between 
$p_k$ and $p_{k+1}$, the equality
$$\gp{x_{k+1}}{x_k}{x_{k+2}} = d_X(x_{k+1}, x_{k+2}) - d_X(x_{k+2}, p_{k+2}) + O(\delta)$$
which, using \eqref{eq:indepk} and \eqref{eq:bounded nbd}, implies 
$$d_X(x_0, gx_0) - 2\gp{x_0}{gx_0}{g^{-1}x_0} = O(\delta)$$
which
contradicts \eqref{eq:hyperbolic} if $C_0$ is large enough.  The case where $p_{k+1}$ 
lies between $p_k$ and $p_{k+2}$ is completely analogous, therefore
the $p_k$ are monotonically ordered on $\gamma$, and so by
\eqref{eq:translation2}
\[ d_X(p_0, p_k) \ge k ( d_X(x_0, g x_0) - 2\gp{x_0}{g x_0}{g^{-1}
  x_0} + O(\delta) ), \]
which implies, by Proposition \ref{prop:npp2}, a similar bound for $d_X(x_0, x_k)$, and so in fact
$(x_k)_{k \in \N}$ is quasi-geodesic, with $\tau(g) \ge d_X(x_0, g x_0) -
2\gp{x_0}{g x_0}{g^{-1} x_0} + O(\delta)$. 

The upper bound on $\tau(g)$ follows from the triangle inequality; indeed, for each $y \in X$ one has 
$$\tau(g) \le d_X(y, gy)$$
and the desired bound follows by taking as $y$ the midpoint of the geodesic segment between $x_0$ and $g x_0$, 
completing the proof of Proposition \ref{prop:hyperbolic}.
\end{proof}



\noindent Joseph Maher \\
CUNY College of Staten Island and CUNY Graduate Center \\
\url{joseph.maher@csi.cuny.edu} \\

\noindent Giulio Tiozzo \\
Yale University\\
\url{giulio.tiozzo@yale.edu} \\


\end{document}